\documentclass[12pt]{article}
\usepackage[left=1in,top=1in,right=1in,bottom=1in,nohead]{geometry}

\usepackage{url,subfig,latexsym,amsmath,amssymb, amsfonts,enumerate,
  epsfig, graphics,times, amsthm,mathtools, bbm}

\usepackage{tikz, pgfplots}
\usetikzlibrary{automata,positioning,arrows, intersections, calc}
\pgfplotsset{compat=1.8}

\usepackage{cite}
\usepackage{wrapfig}
\usepackage{xcolor}
\usepackage{nameref}
\usepackage{listings}
\usepackage[colorlinks]{hyperref}
\usepackage{makeidx}
\usepackage{tabularx}
\usepackage{array}
    \newcolumntype{P}[1]{>{\centering\arraybackslash}p{#1}}
    \newcolumntype{M}[1]{>{\centering\arraybackslash}m{#1}}

    \providecommand{\keywords}[1]{\textbf{Keywords:} #1}
     \providecommand{\classification}[1]{\textbf{AMS subject classifications:} #1}
\newtheorem{thm}{Theorem}[section]
\newtheorem{cor}[thm]{Corollary}
\newtheorem{lem}[thm]{Lemma}

\newtheorem*{thm*}{Theorem}

\theoremstyle{definition}
\newtheorem{defn}{Definition}[section]
\newtheorem{rem}{Remark}[section]
\newtheorem{assumption}{Assumption}[section]
\newtheorem{example}{Example}[section]
\newcommand{\EE}{{\mathcal{E}}}
\newcommand{\gap}{\mathrm{gap}}

\def\S{\mathbb{S}}

\def\T1{T^1_{\{x_n\}}}
\def\Sp{\mathcal S}
\def\C{\mathcal C}
\def\Re{\mathcal R}
\def\R{\mathbb R}

\def\Z{\mathbb Z}
\def\A{\mathcal A}
\def\Td1{T^{D,1}_{\{x_n\}}}
\def\Ts1{T^{S,1}_{\{x_n\}}}

\def\TV{\text{TV}}

\def\K{\mathcal{K}}

\def\Tid1{T^{D,1}_{\{x_n+\zeta_i\}}}
\def\Tis1{T^{S,1}_{\{x_n+\zeta_i\}}}
\def\Tjs1{T^{S,1}_{\{x_n+\zeta_j\}}}
\def\Tjd1{T^{D,1}_{\{x_n+\zeta_j\}}}

\newcount\colveccount
\newcommand*\colvec[1]{
        \global\colveccount#1
        \begin{pmatrix}
        \colvecnext
}
\def\colvecnext#1{
        #1
        \global\advance\colveccount-1
        \ifnum\colveccount>0
                \\
                \expandafter\colvecnext
        \else
                \end{pmatrix}
        \fi
}

\makeatletter
\newcommand{\labitem}[2]{%
\def\@itemlabel{\textbf{#1}}
\item
\def\@currentlabel{#1}\label{#2}}
\makeatother

\title{A new path method for  exponential ergodicity of  Markov processes on $\mathbb Z^d$, with applications to  stochastic reaction networks}

\author{
David F. Anderson\thanks{Department of Mathematics, University of
  Wisconsin, Madison, USA.  anderson@math.wisc.edu, grant support from NSF-DMS-2051498 and Army Research Office grant W911NF-18-1-0324.},
\and 
Daniele Cappelletti \thanks{Department of Mathematics, Politecnico di Torino, Turin, Italy. daniele.cappelletti@polito.it, supported by the MIUR grant ‘Dipartimenti di Eccellenza 2018-2022’
(E11G18000350001).}\and
  Wai-Tong Louis Fan \thanks{Department of Mathematics, Indiana University, USA. waifan@iu.edu, grant support from NSF-DMS-1855417, NSF-DMS-2152103 and Office of Naval Research grant TCRI-N00014-19-S-B001.} \and
Jinsu Kim\thanks{Department of Mathematics, Pohang University of Science Technology, Pohang 37673, Republic of Korea. jinsukim@postech.ac.kr, grant support from the National Research Foundation of Korea (NRF) grant funded by the Korea government (MSIT) (No. 2022R1C1C1008491 and No.2021R1A6A1A10042944), POSCO HOLDINGS research fund (2022Q019), and Samsung Electronics Co., Ltd (IO230407-05812-01).}
  }

\begin{document}

\tikzset{every node/.style={font=\tiny,sloped,above}}
\tikzset{every state/.style={rectangle, minimum size=0pt, draw=none, font=\normalsize}}
 \tikzset{bend angle=15}
 
\maketitle

\begin{abstract}
This paper provides a new path method that can be used to determine when an ergodic continuous-time Markov chain on $\Z^d$ converges exponentially fast to its stationary distribution in $L^2$.  Specifically, we provide general conditions that guarantee the positivity of the spectral gap. Importantly, our results do not require the assumption of time-reversibility of the Markov model.
We then apply our new method to the well-studied class of stochastically modeled reaction networks. Notably, we show that each complex-balanced model that is also ``open'' has a positive spectral gap, and is therefore exponentially ergodic. We further illustrate how our results can be applied for models that are not necessarily complex-balanced. Moreover, we provide an example of a detailed-balanced (in the sense of  reaction network theory), and hence complex-balanced, stochastic reaction network that is not exponentially ergodic.  We believe this to be the first such example in the literature. 
\end{abstract}


\classification{60J27, 60J28}
%\tableofcontents

\section{Introduction}

Consider an ergodic continuous-time Markov chain $X=(X(t))_{t\geq 0}$ with a discrete countable state space $\mathbb{S}\subset \Z^d$ and  a unique stationary distribution $\pi$. For such a process, a natural  question is:  what is the rate of convergence of the distribution of the process to $\pi$?

Let $\pi(f)=\sum_{x\in \mathbb{S}} f(x)\,\pi(x)$ be the expectation of a function $f$ on $\mathbb{S}$ with respect to $\pi$, and let $\|f\|_{L^2(\pi)}=\big(\pi(f^2)\big)^{1/2}$  be the $L^2$-norm.
Denote by $P_t$  the Markov semigroup for $X$.  
That is, for  $f\in L^2(\pi)$,   $P_t f$ is the function on $\mathbb{S}$ whose $x$th element is $P_tf (x) = \mathbb{E}_x[f(X(t)]$. 
Then we  say that the distribution \textit{converges exponentially fast in $L^2$} if there is a constant $C>0$ so that 
 for each compactly supported function $f$ (or, equivalently, for each $f \in L^2(\pi)$) we have  
 \begin{equation}
 \label{eq:needed21}
\left \|P_tf-\pi(f) \right \|_{L^2(\pi)} \leq \| f - \pi(f)\|_{L^2(\pi)} e^{-Ct}, \text{ for all } t \ge 0.
\end{equation}

Of course, we can also have convergence in other norms.  For example, we have exponential convergence in the total variation norm if there is an $\eta >0$ such that for all $x\in \mathbb S$ we have
\begin{align}\label{eq:toprove}
\Vert P^t(x,\cdot) - \pi(\cdot) \Vert_{\TV} \le B(x) e^{-\eta t}, \quad \text{for all } t \ge 0,
\end{align}
for some function $B(x):\mathbb{S} \to \R_{\ge0}$ \cite{down1995exponential},  where $P^t(x,\cdot)$ is the probability distribution of the process at time $t$ with initial condition $x$ and where the total variation distance between two probability measures on a measurable space $(\Omega,\mathcal{F})$ is defined as $\Vert \mu - \nu\Vert_{\TV} := \sup_{A\in \mathcal{F}}|\mu(A) - \nu(A)|$.  In our case of a discrete state space, $\|\mu-\nu\|_{\TV} = \frac12 \sum_{x\in \mathbb{S}} |\mu(x) - \nu(x)|$, and so is an $L^1$ norm \cite{YuvalLevinMixing}.  Hence, $L^2$ exponential convergence immediately implies exponential convergence in total variation via the Cauchy-Schwartz inequality.

A natural, and related, way to quantify the convergence of the distribution is via the mixing time of the process.    Let $\varepsilon \in (0,\tfrac12)$.  The \textit{mixing time} with distance $\varepsilon$ for the process with initial condition $x\in \S$ is
\begin{align}\label{eq:mixing time}
\tau_x^\varepsilon := \inf_{t \ge 0} \{\Vert P^t(x,\cdot) - \pi(\cdot) \Vert_{\TV} \le \varepsilon\}. 
\end{align}
  See \cite{YuvalLevinMixing} for a pedagogical introduction to mixing times.  Hence, when \eqref{eq:toprove} holds, an upper bound for the mixing time follows since $\tau_x^\varepsilon \le \frac1\eta\left[ \ln(B(x) + \ln(\varepsilon^{-1})\right] = O(\ln(B(x)))$.

 For continuous-time Markov chains on a countable state space, common methods to demonstrate exponential ergodicity include (i) the use of Foster-Lyapunov functions \cite{mixing_And_Kim,superlya2012, bakry2008rate, down1995exponential} and (ii) analysis of the spectral gap of the model \cite{bakry2008rate, YuvalLevinMixing, Liggett89, Han16}. We take the second approach and in the first part of this paper we provide new sufficient conditions, in the form of a new path method, that imply that the spectral gap is positive (and, in particular, that the Poincar\'e inequality holds) which in turn implies that \eqref{eq:needed21} holds. This, in turn, implies the mixing time is of order $-\log(\pi(x))$, when $X(0)=x$. These results are located in Theorem \ref{thm:main for poincare} and Corollaries \ref{cor:main} and \ref{cor:relaxed}.  Very loosely, our results apply to models whose stationary distributions have tails that decrease quickly (faster than exponential, such as the Poisson distribution, among others) and that allow for transitions of the form $\pm e_i$, where $\{e_i\}$ is the canonical basis for $\Z^d$ (there are more technical assumptions related to lower bounds for the rates of these transitions).  Importantly, we \textit{do not} assume that the models are time-reversible. That is, we need not have that the stationary distribution $\pi$ satisfies $\pi(x)q(x,z)=\pi(z)q(z,x)$ for all states $x,z$, with $q(\cdot,\cdot)$ denoting the transition rates. In fact, we even allow our models to be irreversible in that a positive rate of transition from state $x$ to state $z$ need not imply there is a positive rate of transition from state $z$ to state $x$.

To illustrate the new path method, we present three key examples. First, in Example \ref{ex:motivation}, a simple birth-death model is provided in order to motivate our new approach.  In Example \ref{ex:key example}, we present a Markov system where the well-known ``canonical path method'' falls short, while our method establishes exponential ergodicity.  Finally, in Example \ref{ex:non expo}, we demonstrate the necessity of the assumption pertaining to transitions of the form $\pm e_i$ for the class of models we are interested in (stochastic reaction networks).

In the second part of this paper, we apply our method to the class of models that motivated this work: stochastic reaction networks.  Stochastic reaction networks are used ubiquitously in the biology literature to model biological interaction systems (which include intracellular processes, viral  infections,  signaling systems, etc.).  	
  Reaction networks are often depicted  via a \textit{reaction graph}, which is a graphical representation of the interactions between the  constituent ``species'' of the model.  See  Figure \ref{figure:EnvZ_model1} for a  network describing the intracellular interactions of the EnvZ/OmpR signal transduction system in \textit{E.~coli} \cite{AEJ2014, Sh-F}.   
      \begin{figure}
\centering
{\includegraphics[width=8cm]{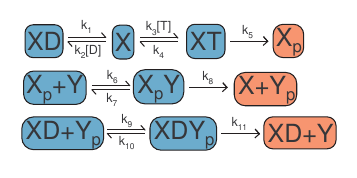}}
\caption{Network for the EnvZ-OmpR signaling system in \textit{E.~coli}.}
	\label{figure:EnvZ_model1}
	\end{figure} 
 The constituent ``species'' are written in black, 
 and the reactions are the transitions implied by the directed graph.  For example, the reaction  $X_P + Y \to X_pY$ uses one $X_P$ molecule and one $Y$ molecule to generate one $X_pY$ molecule.  The rate constants, such as $k_1$, are written next to the arrows.  The standard stochastic model then treats the system as a continuous-time Markov chain on $\Z^d_{\ge 0}$, with $X_i(t)$ giving the count of species $i$ at time $t$.  Transitions for the model are then implied by the reactions.  The standard intensity functions for the transitions (called mass-action kinetics) are polynomial in form.  These models will be introduced in Section \ref{sec:modelsandresults}.

Often, the goal of the mathematical theory for reaction networks is to relate topological conditions of the underlying reaction network (which are easy to check) to the qualitative behavior of the associated dynamical system (which is typically difficult to quantify). In previous research pertaining to the discrete stochastic models, classes of reaction networks characterized by graph topological conditions have been identified for which key long-term behaviors are guaranteed, independent of system parameters. These behaviors include non-explosion \cite{anderson2018non}, positive recurrence \cite{anderson2020stochastically, anderson2020tier, AndProdForm, anderson2018some}, and large deviation principles \cite{agazzi2018large, agazzi2018geometry, anderson2020tier}. However, finding structural classifications for exponential ergodicity has not been widely studied.  The exception is \cite{mixing_And_Kim}, which used Foster-Lyapunov functions to prove exponential ergodicity for a class of stochastic reaction networks. Therefore, the second goal of this paper is to continue to fill this gap in our understanding. 
Specifically,  we utilize our methods via Theorem \ref{thm:main for poincare} and its corollaries to prove exponential ergodicity in $L^2$ for a subclass 
 of stochastic reaction networks that are ``open,''  meaning that each species can enter and leave the system. 
 Notably, our results apply (but are not limited to) \emph{complex-balanced} models. In particular, we are able to state the following result:

\smallskip
{\it An open and complex-balanced stochastic reaction network is $L^2$-exponentially ergodic.}
\smallskip

\noindent
A precise statement is given in Corollary~\ref{cor:cb_inout}, and a more general formulation is given in Corollary~\ref{cor:cb_inout_relaxed}.
 The ``complex-balanced'' condition of reaction network theory is a balance condition similar to the ``traffic conditions'' found in probability theory \cite{Kelly1979}; the concept in the setting of reaction network theory was introduced in the seminal papers \cite{Feinberg72, Horn72, HornJack72} and has been an object of study ever since (for example, see \cite{AndProdForm, AndGAC_ONE2011,CraciunShiu09, CW2016}).

Finally, in Example \ref{ex:non expo} we provide an example of a detailed-balanced, and hence complex-balanced, reaction network (which is time-reversible in the sense of Markov chains) that is not exponentially ergodic.  We believe this to be the first such model analyzed in the literature.

The remainder of the paper is outlined as follows. In Section \ref{sec:main}, we introduce our main technical theorem and its corollaries along with the needed background pertaining to analysis of the spectral gap. In Section \ref{sec:modelsandresults}, we formally introduce stochastic reaction networks and then provide our main results pertaining to this class of models (which motivated the present work). We anticipate interest from a broad range of researchers, including those from the biosciences, and so add necessary technical background material in the appendix.  We also relegate certain proofs from Section \ref{sec:main} to the appendices.

\section{Main analytical results}\label{sec:main}

In Section \ref{sec:models}, we provide a detailed overview of how to apply the analytical methods relating the Poincar\'e inequality to $L^2$ exponential convergence in the current context (continuous-time Markov chain on a discrete state space, with time-reversibility \textit{not} assumed).  Importantly, in our Assumption \ref{intcnd} we will carefully specify the regularity conditions needed on the Markov model that guarantee the objects under consideration are well-defined and allow the analytical methods to go through.  While much of this material is, to a greater or lesser extent,  known, see for example \cite{bakry2014analysis, MR1106707, Liggett89, rockner2001weak, Han16}, all the references we are aware of on the subject make assumptions that we do not, with the most notable being the common assumption that the Markov process is time-reversible. We therefore provide our own proofs for the preliminary results.  We note that our proofs are somewhat similar to those found in the concise lecture notes by Ramon van Handel \cite{Han16}. 
However, in our opinion there are important technical differences, especially related to when certain limits and sums can be exchanged.
 Still, in the interest of readability we have placed many of the proofs from this section in the appendices.

Next, in Section \ref{sec:main_result}, we provide the statement and proof of our first main result, which provides sufficient conditions for $L^2$ exponential ergodicity of continuous-time Markov chains  on $\Z^d$.  In Section \ref{sec:corrolaries}, we provide corollaries that are easier to apply  for continuous-time Markov chains on $\Z^d_{\geq0}$.  Finally, in Section \ref{subsec:canonical}, we provide two important examples.  In the first example, we demonstrates why new theory was needed by showing that the canonical path method stands silent on  certain models of interest, while our new theory is strong enough to prove exponential ergodicity.  In the second example, we demonstrate the necessity of one of our main structural/analytic assumptions found in the corollaries: that transitions of the form $\pm e_i$ have rates that are lower bounded (together with other technical assumptions).  Moreover, we believe that this is the first example in the literature of a complex-balanced (and detailed-balanced) stochastic reaction network that is shown to be not exponentially ergodic  in $L^2$. 

\subsection{Models and methods}
\label{sec:models}

\subsubsection{The basic model}
We assume basic knowledge of the theory of continuous-time Markov chains, as  can be found, for example, in \cite{NorrisMC97}.
For some positive integer $d$, we let $X$ be a continuous-time  
Markov chain with a countable state space $\S \subset \Z^d$ and stationary distribution $\pi$. For $x, z\in \mathbb{S}$ with $x\neq z$, we denote by  $q(x,z)$  the transition rate from state $x$ to $z$; that is, 
\begin{equation}\label{qxz}
    \mathbb{P}(X(t+\Delta t)=z \ | \ X(t)=x)=q(x,z)\Delta t + o(\Delta t),  \quad \text{as}\quad  \Delta t \to 0.
\end{equation}
We will denote $q_x  = \sum_{z\in \mathbb{S}\setminus \{x\}} q(x,z)$ and $q(x,x) = -q_x$.  
We  will not assume that $X$ is time-reversible.  That is, we allow  the situation $\pi(x) q(x,z)\neq \pi(z) q(z,x)$ for some $x,z\in \S$. In fact, we even allow for the state space to be irreversible, that is, $q(x,z)>0$ and $q(z,x)=0$ for some $x,z\in \S$.

We will assume throughout that our collection of rates, $\{q(x,z)\}$, satisfies the following regularity assumption.

\begin{assumption}\label{intcnd}
The collection of rates $\{q(x,z)\}$ satisfies the following four conditions. 
\begin{enumerate}[(i)]
\item For $x,z\in \mathbb{S}$, with $x\neq z$, we have $0\leq q(x,z)<\infty$.
\item For each $x\in \mathbb{S}$, we have $q_x = \sum_{z\in \mathbb{S}\setminus\{x\}}q(x,z)<\infty$.
\item For each $z\in \mathbb{S}$, we have $\sup_{x\in \mathbb{S}\setminus\{z\}}q(x,z)<\infty$.
\item $X$ is an irreducible and non-explosive continuous-time Markov chain on a countable state space $\S\subset \Z^d$, with unique stationary distribution $\pi$.
\hfill $\triangle$
\end{enumerate}
\end{assumption}

The first condition is standard and states that all rates are non-negative and bounded.  The second condition guarantees that the holding time in each state is positive with probability one (no instantaneous states). The third condition is more technical in nature, but is required, for example, in the proof of Lemma \ref{lemma:compact_f}. Finally, the last condition guarantees that the distribution of the process $X$ is uniquely determined by the transition rates $\{q(x,z)\}$, and that the distribution of $X(t)$ converges to $\pi$, as $t$ goes to infinity. The most common way to ensure that condition (iv) holds is to assume that $\sum_x \pi(x) q_x < \infty$ \cite{Nonexplo}. Assumption~\ref{intcnd} ensures that the Kolmogorov forward equations 
\begin{equation}\label{eq:kolmogorov_forward}
\frac{d P_x(X(t)=y)}{dt}=\left(\sum_{z\in\S\setminus\{y\}}q(z,y)P_x(X(t)=z)\right) -q_yP_x(X(t)=y)
\end{equation}
hold for all $x,y\in \S$ and $t\in [0,\infty)$, where $P_x(\cdot):=P(\cdot|X(0)=x)$ (see the appendix in \cite{Nonexplo}). As a consequence, we have that $\pi$ satisfies
\begin{equation}\label{eq:equilibrium}
\pi(y)\sum_{z\in \mathbb{S}\setminus\{y\}}q(y,z)=\sum_{z\in \mathbb{S}\setminus\{y\}}\pi(z)q(z,y).
\end{equation}

Denote by $\mathcal{A}$ and $Dom(\mathcal{A})$  the infinitesimal generator of $X$ and its domain in $L^2(\pi)$, respectively (see  \cite[(Chapter 1)]{ethier2009markov} for relevant definitions).  Then for $f \in Dom(\mathcal{A})$,
\begin{equation}\label{Def:Generator}
\mathcal{A} f(x)=\sum_{z\in\S\setminus\{x\}}(f(z)-f(x))\,q(x,z).
\end{equation}
Let $\mathcal{B}_c$ be the collection of compactly supported functions on $\S$. The following technical lemma demonstrates the usefulness of Assumption \ref{intcnd}.  For completeness, we placed the proof in Appendix \ref{appendix:proof_1234}.

\begin{lem}\label{lemma:compact_f}
Let  $X$ be a continuous-time Markov chain on $\mathbb{S} \subset \Z^d$ with transition rates $\{q(x,z)\}$ that satisfy Assumption~\ref{intcnd}.  
Then, $\mathcal{B}_c \subset Dom(\mathcal{A})$ and
\begin{equation}\label{eq:tocheck323}
\|\mathcal A f\|_\infty < \infty \qquad \text{for }f\in \mathcal{B}_c.
\end{equation}
\end{lem}

\begin{rem}
 The proof of Lemma~\ref{lemma:compact_f} still holds if we only assume points (i)-(iii) in Assumption~\ref{intcnd}, provided that we consider $X$ to be the \emph{minimal} process, i.e.\ the process that is absorbed in a cemetery state $\infty$ after a potential explosion. In this case, \cite[Theorem 4]{Nonexplo} can still be applied and the result be derived.\hfill $\dagger$
\end{rem}

\subsubsection{Methods: Poincar\'e inequalities and exponential \texorpdfstring{$L^2$}{Lg} convergence}

We make precise some notation from the introduction and introduce some new ones. For the stationary distribution $\pi$  let    
\[
    \langle f,\,g \rangle_{L^2(\pi)}:=\sum_{x\in \S}f(x)g(x)\pi(x)
\]
be the inner product of the $L^2$-space $L^2(\pi)$.  We let $\pi(f)=\sum_{x\in \mathbb{S}} f(x)\,\pi(x)$ be the expectation of $f$,
$\|f\|_{L^2(\pi)}=\big(\pi(f^2)\big)^{1/2}$  the $L^2$-norm,
and
$Var_{\pi}(f)=\pi (f^2) - \left(\pi (f)\right)^2= \left\|f-\pi (f) \right\|^2_{L^2(\pi)}$
 the variance, all with respect to  $\pi$. 

Denote by $P_t$  the Markov semigroup for the process $X$.  
%That is, for  $f\in L^2(\pi)$,   $P_t f$ is the function on $\mathbb{S}$ whose $x$th element is $P_tf (x) = \mathbb{E}_x[f(X(t)]$.  
Recall that if $f \in Dom(\mathcal{A})$, then $P_t f\in Dom(\mathcal{A})$ and, moreover,  $\mathcal{A} P_t f=P_t \mathcal{A} f$ \cite[Proposition 1.5]{ethier2009markov}.

For $f \in L^2(\pi)$ and $g\in Dom(\mathcal{A})$, we
define
\begin{align}
\EE(f,g):=&
-\langle f,\,\mathcal{A}g \rangle_{L^2(\pi)} = -\sum_{x,z\in \S} f(x)\,\big(g(z)-g(x)\big)\,\pi(x)\,q(x,z). \label{Def:DFa}
\end{align}
The bilinear operator $\EE$ is referred to as the Dirichlet form of the process $X$ (see \cite[Chapter 1]{ma2012introduction}).  Of particular interest will be $\EE(f,f)$ for $f$ from a rich enough class of functions.
A slightly different representation will be useful.  So, for $f \in Dom(\mathcal{A})$,  let 
\begin{align}\label{eq:78657857}
\EE^*(f):=&\frac{1}{2}\sum_{x,z\in \mathbb{S}}\big(f(x)-f(z)\big)^2\,\pi(x)\,q(x,z).
\end{align}

The proof of the following lemma is found in Appendix \ref{appendix:proof_9808097}.

\begin{lem}\label{lem:eff}
Suppose that $X$ satisfies Assumption~\ref{intcnd}.
For  $f\in Dom(\mathcal{A})$ with $\|f\|_\infty<\infty$, both the sum \eqref{Def:DFa} for $\EE(f,f)$ and the sum \eqref{eq:78657857} for $\EE^*(f)$ converge absolutely.  Also, $\EE(f,f)=\EE^*(f)$.
\end{lem}

 With Lemma \ref{lem:eff} in hand, we define the \emph{spectral gap} as   
\begin{align}\label{Def:GapA}
\gap(\EE) 
&:= \inf\left\{ \frac{\EE(f,f)}{Var_{\pi}(f)}:\;f\in Dom(\mathcal{A}), \|f\|_\infty <\infty ,Var_{\pi}(f)\neq 0 \right\}.
\end{align}
Note that  $\gap(\EE)> 0$  if and only if there is a $C>0$ for which 
\begin{equation}\label{PI}
  \EE(f,f) \geq  C\, Var_{\pi}(f)
\end{equation}
for all bounded $f\in Dom(\mathcal{A})$.   The inequality \eqref{PI} is sometimes called the \textit{Poincar\'e inequality}. In light of Lemma \ref{lem:eff}, we have the following alternative representation for $\text{gap}(\EE)$,
\begin{align}\label{eq:2222NEW}
    \text{gap}(\EE) = \inf\left\{\EE^*(f) : f \in Dom(\A), \|f\|_\infty < \infty,  \pi(f) = 0, \pi(f^2) = 1\right\},
\end{align}
which will be useful to us throughout.

The following theorem is key to our analysis, and explains our interest in the spectral gap. In short, the result states that in our setting $\text{gap}(\EE)>0$ implies exponential $L^2$ convergence of the distribution of $X(t)$ to $\pi$, as $t$ tends to infinity. 
As mentioned in the preamble to this section, the result is considered to be well-known. However, we were not able to find a version exactly fitting our needs.   
We also stress again that the unpublished lecture notes from Ramon van Handel has a proof that is very similar to ours (though there are important differences on the justification for passing certain limits through sums).  In particular, see \cite[Theorem 2.18 and Remark 2.19]{Han16}.

\begin{thm}\label{thm:adlkjfkd}
Suppose Assumption~\ref{intcnd} is satisfied.  If the process $X$ is time-reversible, then the following statements are equivalent.
\begin{enumerate}[(i)]
    \item $\gap(\EE) \geq  C$ for some $C\in(0,\infty)$.
    \item For all $f\in \mathcal{B}_c$ and $t\geq 0$,    \begin{equation}\label{ExpoEr2}
\left \|P_tf-\pi(f) \right \|_{L^2(\pi)} \leq e^{-2Ct}\, \left \|f-\pi(f) \right \|_{L^2(\pi)}.
\end{equation}
\item   For  $f\in \mathcal{B}_c$ there is a constant $C_f$ so that for all $t\geq 0$,    \begin{equation}\label{ExpoEr3}
\left \|P_tf-\pi(f) \right \|_{L^2(\pi)} \leq C_f e^{-2Ct}.
\end{equation} 
\end{enumerate}
Moreover, if the process is not time-reversible, we  have the implications $(i) \iff (ii)\implies (iii)$. 
\end{thm}

\begin{proof}
We do not assume that the process is time-reversible and begin by showing $(i) \implies (ii)$.

Suppose $\gap(\EE) \geq  C$ for some $C\in(0,\infty)$. Let $f\in \mathcal{B}_c$ be arbitrary.  We have $f\in Dom(\mathcal{A})$ by Lemma \ref{lemma:compact_f}.  Because $f\in Dom(\mathcal{A})$, \cite[Proposition 1.5]{ethier2009markov} implies that $P_t f \in  Dom(\mathcal{A})$ for all $t \ge0$. Recalling that $P_tf (x) = \mathbb{E}_x[f(X(t))]$, we immediately have $\|P_t f\|_\infty   \le \|f\|_\infty < \infty$. Hence, we can  consider $\EE(P_tf,\,P_tf)$ and we note that by our assumption that $\gap(\EE)\ge C$,  we have that 
\begin{align}\label{eq:gap1}
    \EE(P_tf,\,P_tf) \geq C\,Var_{\pi}(P_tf),
    \end{align}
    for all  $t\geq 0$. 

Observe that $\pi(P_t f) = \pi (f)$ is constant in $t$.  Hence,
\begin{align}
\frac{d}{dt}Var_{\pi}(P_tf)
&= \frac{d}{dt} \left\{\pi((P_tf)^2)- [\pi(P_tf)]^2 \right\}= \frac{d}{dt} \pi((P_tf)^2)\notag\\
&= \frac{d}{dt} \sum_{x\in \mathbb{S}} \pi(x) (P_t f(x))^2. \label{eq:dkfjadkj}
\end{align}
Next,  and again by \cite[Proposition 1.5]{ethier2009markov}, $\frac{d}{dt} P_t f = \mathcal{A}P_t f$.  Hence, the series giving the  term-wise derivatives of the final line above is
\begin{align*}
 \sum_{x\in \mathbb{S}} \pi(x) \frac{d}{dt}(P_t f(x))^2= 2\sum_{x\in\mathbb{S}}\pi(x)  P_tf(x)\,\mathcal{A}P_tf(x) =2 \langle P_tf,\,\mathcal{A}P_tf \rangle_{L^2(\pi)}. 
\end{align*}
Because $f\in \mathcal{B}_c$, we know $\|f\|_\infty<\infty$.  Hence, by Lemma \ref{lem:uniformconv} the sum above converges uniformly for $t\ge 0$.  Thus, we may pass the derivative through the sum in \eqref{eq:dkfjadkj} and conclude
\begin{align}
    \frac{d}{dt}Var_{\pi}(P_tf) 
&=2\langle P_tf,\,\mathcal{A}P_tf \rangle_{L^2(\pi)}\notag \\
&= -2\,\EE(P_tf,P_t f)\label{eq:7896786}\\
&\leq  -2 C Var_{\pi}(P_tf),\notag
\end{align}
for $t \ge 0$, where the final inequality follows from \eqref{eq:gap1}.
Hence for all $t \ge0$, we have  
\[
Var_{\pi}(P_tf) \leq e^{-2Ct} Var_{\pi}(f),
\]
which is exactly \eqref{ExpoEr2}.

The proof that $(ii) \implies (i)$ is now straightforward by using \eqref{eq:7896786} (see \cite[Proof of Proposition 2.18]{Han16}). 
The fact that $(iii)\implies (i)$ (and $(ii)$) in the time-reversible setting follows from a convexity argument. See the proof of Theorem 2.18 in \cite{Han16} for these details.
\end{proof}

The next lemma formally connects the previous result to the convergence of the distribution in the total variation norm, and hence makes the connection to mixing times.

\begin{lem}\label{lem:randome11}
Suppose $X$ is an irreducible and positive-recurrent continuous-time Markov chain with countable state space $\S\subset \Z^d$ and stationary distribution $\pi$.  Suppose that there is a $C>0$ for which 
\begin{equation}\label{ExpoEr23}
\left \|P_tf-\pi(f) \right \|_{L^2(\pi)} \leq e^{-Ct}\, \left \|f-\pi(f) \right \|_{L^2(\pi)}
\end{equation}
for all $f\in \mathcal{B}_c$ and $t\geq 0$. Then
\begin{align*}%\label{eq:L2 and TV}
\Vert P^t(x,\cdot) - \pi(\cdot) \Vert_{\text{TV}} \le \frac{2}{\pi(x)} e^{-Ct}.
\end{align*}
\end{lem}

\begin{proof}
Let $K$ be a fixed, compact subset of $\mathbb{S}$ and let $f(x) = \mathbbm{1}_K(x)$.  For any $x \in \mathbb{S}$, we have
\begin{align} 
    \pi(x)|P^t&(x,K)-\pi(K)| \le \sum_{z\in \mathbb{S}}\pi(z)|P^t(z,K)-\pi(K)| \notag\\
    &\le  \left(\sum_{z\in \mathbb{S}}\pi(z)\right)^{1/2}\left( \sum_{z\in \mathbb{S}} \pi(z) |P^t(z,K)-\pi(K)|^2\right)^{1/2} \tag{Cauchy-Schwarz}\\
    &= 1\cdot \| P_t f - \pi(f) \|_{L^2(\pi)}\notag\\
    &\le 2e^{-Ct},\notag
\end{align}
where the final inequality follows from \eqref{ExpoEr23} and the fact that $\|f-\pi(f)\|_{L^2(\pi)}\le 2$ for our choice of $f$. 
  Hence, taking the sup over the compact sets $K$,  we have 
\begin{align*}
\Vert P^t(x,\cdot) - \pi(\cdot) \Vert_{\TV} \le \frac{2}{\pi(x)}e^{-Ct},
\end{align*}
and the result is shown.
\end{proof}

Note  that the conclusion of Lemma \ref{lem:randome11}  gives an upper bound for 
the mixing time: if the assumptions of Theorem~\ref{thm:adlkjfkd} are satisfied, then
\begin{align}\label{eq:mx time from l2 norm}
\tau_\varepsilon^x&=\inf\{t>0: \Vert P^t(x,\cdot) - \pi(\cdot)\Vert_{\TV} \leq \varepsilon\}\notag\\
&\leq \frac{1}{\gap(\EE)}\left(-\ln\left(\frac{\varepsilon}{2}\right)-\ln(\pi(x))\right) = \frac{1}{\gap(\EE)}\left(\left|\ln\left(\frac{\varepsilon}{2}\right)\right| + |\ln(\pi(x))|\right).
\end{align}

\subsection{Statement and proof of our main result}
\label{sec:main_result}

In this section we again consider a non-explosive, irreducible continuous-time Markov chain $X$ on $\mathbb{S} \subset \Z^d$ with transition rates $\{q(x,z)\}$.  Theorem \ref{thm:main for poincare} below, which is our main technical result of the first part of our paper, gives sufficient conditions, in terms of  the stationary distribution $\pi$ of the associated Markov chain and the transition rates $\{q(x,z)\}$, for exponential ergodicity of the Markov chain. 

As before, we will assume throughout that our process satisfies the mild regularity Assumption \ref{intcnd}. For Theorem \ref{thm:main for poincare}, we need a new assumption. We begin with a definition.

\begin{defn}\label{Def:active}
Let $(x,x')\in \mathbb{S}\times \mathbb{S}$ be an ordered pair of points. 
A \emph{directed path} $\gamma(x,x')$ from $x$  to $x'$ is an ordered list of distinct states $(x_1,\dots, x_L)$, with each $x_i \in \mathbb{S}$, such that $x_1=x$ and $x_L=x'$. If $z = x_i$ and $z' = x_{i+1}$ for some $i \le L-1$, then we say $(z,z')$ is an \textit{edge} or sometimes a \textit{directed edge} in $\gamma(x,x')$. We write $(z,z')\in \gamma(x,x')$ if $(z,z')$ is an edge in the path $\gamma(x,x')$ and write $z\in \gamma(x,x')$ if $z = x_i$ for some $1\leq i\leq L$.
We denote by $|\gamma(x,x')|=L$ the \emph{length} of the path.
We say this path is \emph{active} if $q(x_i, x_{i+1}) >0$ for each $i\in \{1,2,\dots,L-1\}$.
\hfill $\triangle$
\end{defn} 

Note that we allow for the possibility that $x'=x$ in the definition above, in which case $\gamma(x,x) = (x)$ and $|\gamma(x,x)| = 1$.  Note also that in the case that $x = x'$ the path is, by definition, active.

 Our new assumption--Assumption \ref{assump:main} below-- involves the existence of very particular sets of  paths.  It is likely that the conditions on the paths will be difficult to digest at first glance,  so we first motivate with an example.  The detailed analysis of this example will also make the proof of Theorem \ref{thm:main for poincare} easier to parse as it contains many of the key steps.

\begin{example}\label{ex:motivation}
    Consider the birth and death model on $\{0,1,2,\dots\}$ with transition rates 
    \[
        q(x,x+1) = 1 \quad \text{and} \quad q(x,x-1)=x \quad \text{for } x \ge 0.
    \]
    This model has infinitesimal generator 
    \[
        \mathcal{A}f(x) = f(x+1)-f(x) + x(f(x-1)-f(x))
    \]
    and stationary distribution $\pi(x)=\frac{e^{-1}}{x!}$ (This model corresponds to the reaction network $\emptyset \overset{1}{\underset{1}{\rightleftarrows}} X_1$, or the M/M/$\infty$ queue).
    
    While it is straightforward to prove exponential ergodicity for this model via the use of Foster-Lyapunov functions, we want to instead prove positivity of the spectral gap of the process so as to motivate our constructions.  Hence, we want to show the existence of a $C>0$ so that 
\[
    Var_\pi(f) \le C \, \mathcal{E}^*(f).
\]
for any $f\in Dom(\mathcal{A})$ with $\|f\|_\infty<\infty$.
To do so, we start with the following  bound:
\begin{align*}
    Var_\pi(f) &= \sum_{x=0}^\infty \left(f(x)-\sum_{x'=0}^\infty f(x')\pi(x')\right)^2 \pi(x)= \sum_{x=0}^\infty \left(\sum_{x'=0}^\infty (f(x)-f(x'))\pi(x')\right)^2 \pi(x)\\
    &\le  \sum_{x=0}^\infty \sum_{x'=0}^\infty(f(x)-f(x'))^2\pi(x)\pi(x')\\
    &= 2 \sum_{x=0}^\infty \sum_{x'=x+1}^\infty(f(x)-f(x'))^2\pi(x)\pi(x').
\end{align*}
 To bound the final double sum above by $\mathcal E^*(f)$, we connect the points $x<x'$ via the active path 
 \[
    \gamma(x,x')=(x,x+1, \dots, x')
    \]
    so that $x_i = x+i-1$, for $i = 1,\dots, x'-x+1$.
    Then, by considering a telescoping sum for the term $f(x)-f(x')$ over this path, we have
\begin{align}
& \sum_{x=0}^\infty \sum_{x'=x+1}^\infty(f(x)-f(x'))^2\pi(x)\pi(x') \notag \\
&= \sum_{x=0}^\infty \sum_{x'=x+1}^\infty \left (\sum_{i=1}^{x'-x}(f(x_i)-f(x_{i+1}) \right )^2\pi(x)\pi(x') \label{eq:1dim1} \\
&\le \sum_{x=0}^\infty \sum_{x'=x+1}^\infty |x'-x|\pi(x)\pi(x') \sum_{i=1}^{|x'-x|} (f(x_i)-f(x_{i+1}))^2 \notag \\
&\le \sum_{x=0}^\infty \sum_{x'=x+1}^\infty |x'-x| \frac{\pi(x)\pi(x')}{\displaystyle \min_{i\in \{1,\dots,|x'-x|\}} \pi(x_i)q(x_i,x_{i+1})}\sum_{i=1}^{x'-x} (f(x_i)-f(x_{i+1}))^2 \pi(x_i) q(x_i,x_{i+1})\label{eq:1dim2},
\end{align}
where we used $(\sum_{i=1}^n a_i)^2 \le n\sum_{i=1}^n a_i^2$ in the first inequality above.
The key observation, and the motivation for the calculations above, is that for each pair $x,x'$ the final (third) sum in \eqref{eq:1dim2} is bounded by $2 \mathcal{E}^*(f)$,
\[
 \sum_{i=1}^{x'-x} (f(x_i)-f(x_{i+1}))^2 \pi(x_i) q(x_i,x_{i+1}) \le \sum_{x =0}^\infty \sum_{z =0}^\infty \big(f(x)-f(z)\big)^2\,\pi(x)\,q(x,z)= 2\mathcal{E}^*(f).
\]
Hence, if we can show that
\begin{equation}\label{adfjadlk;fjiouy}
\sum_{x=0}^\infty \sum_{x'=x+1}^\infty |x'-x| \frac{\pi(x)\pi(x')}{\displaystyle \min_{i\in \{1,\dots,|x'-x|\}}\pi(x_i)q(x_i,x_{i+1})}< \infty,
\end{equation}
we will have our desired result.
However, it turns out that this term is actually infinity for the present example and our bounds, especially that of \eqref{eq:1dim2}, were too crude and will require modification (which, in turn, leads to the conditions of Assumption \ref{assump:main}). 
To show that the term on the left of \eqref{adfjadlk;fjiouy} is infinity, we simply make the following observations, which are unique to this example:
\begin{itemize}
    \item $q(x_i,x_{i+1} ) = 1$ for all $i$,
    \item because $x'>x$, we know that $\displaystyle \min_{i \in \{1,\dots, |x'-x|\}} \pi(x_i) q(x_i,x_{i+1}) = \pi(x'-1) = e^{-1}\frac1{(x'-1)!}$.
\end{itemize}
Hence, 
\begin{align*}
    \sum_{x=0}^\infty \sum_{x'=x+1}^\infty |x'-x| \frac{\pi(x)\pi(x')}{\displaystyle \min_{i\in \{1,\dots,|x'-x|\}}\pi(x_i)q(x_i,x_{i+1})} = \sum_{x=0}^\infty \sum_{x'=x+1}^\infty (x'-x) \pi(x)\cdot \frac{1}{x'}=\infty,
\end{align*}
and our first attempt to prove that the gap is positive has failed.

To fix the method for this example we note that if we had changed things so that a modified version of the term 
\[
\frac{\pi(x')}{\displaystyle \min_{i\in \{1,\dots,|x'-x|\}}\pi(x_i)q(x_i,x_{i+1})}
\]
had been of order $\frac{1}{(x')^3}$ (or smaller), then the double sum in \eqref{adfjadlk;fjiouy} would have been finite.  Our strategy, therefore, will be to make the denominator larger by forcing the minimum to be over paths that do not go all the way to $x'$; instead, the paths will stop at $x'-3$.  We will then consider the remaining path from $x'-3$ to $x'$ separately.

Therefore, we begin by defining a function, $t: \{0,1,\dots\} \to \{0,1,\dots\}$ as follows:
\begin{align*}
 t(x)=
 \begin{cases}
     x-3 \quad &\text{for $x> 3$}\\
     x &\text{otherwise.}
 \end{cases}   
\end{align*}
Then, for each pair of points $x$ and $x'$ with $x<x'$, we  consider three paths that connect $x$ and $x'$:
\begin{enumerate}
    \item A path from $x$ to $t(x)$:
    \begin{align}
    \label{eq:009877}
    \gamma(x,t(x)) = \begin{cases}
        (x,x-1,x-2,x-3) & \text{ if $x>3$}\\
        (x) & \text{ else.}
    \end{cases},
    \end{align}
\item A monotonic path from $t(x)$ to $t(x')$: $\gamma(t(x),t(x'))=(t(x),\dots, t(x'))$,
\item A path from $t(x')$ to $x'$, which is simply the reverse of the path $\gamma(x',t(x'))$, defined similarly as above.
\end{enumerate}
We may now perform a telescoping sum similar to what we did in \eqref{eq:1dim1},  but over the three different paths as well.  This leads to terms of the following form:
\begin{align}
& 3\sum_{x=0}^\infty \sum_{x'=x+1}^\infty  \left ( \sum_{i=1}^{|\gamma(x,t(x))|-1}(f(x_{i+1})-f(x_{i}))^2 + \sum_{i=1}^{|\gamma(x',t(x'))|-1}(f(x'_{i+1})-f(x'_{i}))^2 \right ) \pi(x)\pi(x') \label{eq:tele_motivate_example1}  \\
& \sum_{x=0}^\infty \sum_{x'=x+1}^\infty\left ( |\gamma(t(x),t(x'))|\sum_{i=1}^{|\gamma(t(x),t(x'))|-1}(f(z_{i+1})-f(z_{i}))^2 \right ) \pi(x)\pi(x'),\label{eq:tele_motivate_example2}
\end{align}
where $x_i$ and $x_i'$ denote the $i$th points in the paths $\gamma(x,t(x))$ and $\gamma(x',t(x'))$, respectively,  $z_i$ denotes the $i$th point in the path $\gamma(t(x),t(x'))$, and where the $3$ in \eqref{eq:tele_motivate_example1} is an upper bound on each of $|\gamma(x,t(x'))|$ and $|\gamma(x,t(x'))|$. 

We begin with the second term \eqref{eq:tele_motivate_example2}, and bound it in a similar manner as in \eqref{eq:1dim2}:
\begin{align*}
    \sum_{x=0}^\infty &\sum_{x'=x+1}^\infty\left ( |\gamma(t(x),t(x'))|\sum_{i=1}^{|\gamma(t(x),t(x'))|-1}(f(z_{i+1})-f(z_{i}))^2 \right ) \pi(x)\pi(x') \\
    &\le \sum_{x=0}^\infty \sum_{x'=x+1}^\infty |x'-x| \frac{\pi(x)\pi(x')}{\displaystyle \min_{i\in \{1,\dots,|t(x')-t(x)|\}} \pi(z_i)q(z_i,z_{i+1})}\sum_{i=1}^{|\gamma(t(x),t(x'))|-1} (f(z_i)-f(z_{i+1}))^2 \pi(z_i) q(z_i,z_{i+1}),
\end{align*}
where we also used that $|\gamma(t(x),t(x'))| \le |x'-x|$.
Similar to \eqref{eq:1dim2}, the final sum is uniformly bounded by $2\mathcal{E}^*(f)$.  Hence, the  question is whether or not the double sum
\begin{equation}
     \sum_{x=0}^\infty \sum_{x'=x+1}^\infty (x'-x) \frac{\pi(x)\pi(x')}{\displaystyle \min_{i\in \{1,\dots,|t(x')-t(x)|\}} \pi(z_i)q(z_i,z_{i+1})}
     \label{eq:toshow2323}
\end{equation}
is finite.  Now that the minimum is over the path $\gamma(t(x),t(x'))$ (which ends at $t(x') = x'-3$ so long as $x'\ge 4$) we see that
\[
\frac{\pi(x')}{\displaystyle \min_{i\in \{1,\dots,|t(x')-t(x)|\}} \pi(z_i)q(z_i,z_{i+1})} =\frac{\pi(x')}{\displaystyle \min_{i\in \{1,\dots,|t(x')-t(x)|\}} \pi(z_i)}\le \frac{c_0}{(x')^4}, 
\]
for some $c_0>0$.  Hence, \eqref{eq:toshow2323} is indeed finite.

The sums in \eqref{eq:tele_motivate_example1} will be handled in a similar (and simpler) manner.  By symmetry in $x$ and $x'$, we just need the existence of a universal $c_1>0$ for which 
\begin{align*}
    \sum_{x=0}^\infty \sum_{x'=x+1}^\infty  \left ( \sum_{i=1}^{|\gamma(x,t(x))|-1}(f(x_{i+1})-f(x_{i}))^2  \right ) \pi(x)\pi(x') <c_1 \mathcal{E}^*(f).
\end{align*}
Proceeding in a similar manner as above, we  note the following
\begin{align*}
     \sum_{x=0}^\infty \sum_{x'=x+1}^\infty & \left ( \sum_{i=1}^{|\gamma(x,t(x))|-1}(f(x_{i+1})-f(x_{i}))^2  \right ) \pi(x)\pi(x') \le  \sum_{x=0}^\infty \sum_{i=1}^{|\gamma(x,t(x))|-1}(f(x_{i+1})-f(x_{i}))^2 \pi(x)\\
     &=  \sum_{x=0}^\infty \sum_{i=1}^{|\gamma(x,t(x))|-1}(f(x_{i+1})-f(x_{i}))^2 \frac{q(x_i,x_{i+1})\pi(x_i)}{q(x_i,x_{i+1})\pi(x_i)} \pi(x)\\
     &\le \left(\sup_{x\in \{0,1,2,\dots\}} \frac{\pi(x)}{\min\{\pi(x_i)\}} \right) \sum_{x=0}^\infty \sum_{i=1}^{|\gamma(x,t(x))|-1}(f(x_{i+1})-f(x_{i}))^2 q(x_i,x_{i+1})\pi(x_i),
\end{align*}
However, because the path $\gamma(x,t(x))$ moves towards zero from $x$, we have that $\pi(x_i)\ge \pi(x)$ for all $x$ and all associated $x_i$.  Hence, the supremum term is equal to 1.  Next, note that because each edge $(x_i,x_{i+1})$ can only be in at most 3 paths of the form $\gamma(x,t(x))$, we have that the double sum above is bounded by 
\[
3\cdot 2\mathcal{E}^*(f),
\]
where the 2 arises because of the $\frac12$ term in the definition of $\mathcal{E}^*(f)$.  

Collecting all the above, we have shown the existence of a constant $C>0$ for which $Var_\pi(f) \le C \ \mathcal{E}^*(f)$ and the analysis of this example is complete.
\hfill  $\triangle$
\end{example}

  We now introduce our main new assumption, Assumption \ref{assump:main}, on the model. This new assumption is motivated by  Example \ref{ex:motivation}.
Part 1 says that each state can reach a ``nice'' (terminal) state in a finite number of steps (and that number of steps is bounded).  Comparing to Example \ref{ex:motivation}, this will involve the paths defined in \eqref{eq:009877} . Part 2 says that a certain sum, \eqref{eq:sum_condition} below, taken over all pairs of points in $\mathbb{S}$, is finite.  This corresponds to the double sum \eqref{eq:toshow2323} in Example \ref{ex:motivation} being finite.  The key thing to observe is that the sum explicitly uses a mix of regular states (denoted $x$ and $x'$) and terminal states (denoted $t(x)$ and $t(x')$, respectively).
Continuing with the comparison to Example \ref{ex:motivation}, the term $\overline L$ (the max length of $\gamma(x,t(x))$) below is equal to 3 in the example, the term $\overline M$ (representing how many paths of the form $\gamma(x,t(x))$ a particular edge can be in) is also equal to 3, and the term \eqref{eq:supmaxassump} is equal to 1 (as mentioned in the last part of the analysis of the example).

For ease of notation, for a given path $\gamma$ we let $\pi_{\text{min}}(\gamma):=\min_{z\in \gamma} \pi(z)$ be the minimal weight among all states on $\gamma$.

\begin{assumption}\label{assump:main}
The process $X$ with rates $\{q(x,z)\}$ and stationary distribution $\pi$ satisfies the following conditions.

\begin{enumerate}
    \item (Finite steps to a terminal state.) For each $x \in \mathbb{S}$ there is a fixed (terminal) state $t(x) \in \mathbb{S}$ and an active directed path $\gamma_x=\gamma(x,t(x))$.  The following conditions hold for this set of paths
\begin{align}   
&\sup_{x \in \mathbb{S}} |\gamma_x| := \overline L < \infty, \label{eq:overlineL}\\
&\sup_{(x,x') \in \mathbb{S}\times \mathbb{S}} 
| \{ z\in \mathbb{S} :  (x,x') \text{ is an edge in } \gamma_z \}| := \overline M < \infty \label{eq:no infinitely many paths intersects}\\
%& \sup_{x\in \mathbb S}\sup_{x_i\in \gamma_x}|\{z : x_i\in \gamma_z\}|:=\overline M < \infty, \text{ and}\label{eq:no infinitely many paths intersects}\\
& \sup_{x\in \mathbb{S}} \frac{\pi(x)}{\pi_{\text{min}}(\gamma_x)}< \infty.\label{eq:supmaxassump}
\end{align}
Denote the set of terminal states by $\mathcal{T} = \{t(x): x\in \mathbb{S}\}$.  
\item (Fast decay of $\pi$.) There is another set of  paths, 
\[
\Gamma_{\mathcal{T}} =  \{\gamma(t(x),t(x')) \ \text{ or }\  \gamma(t(x'),t(x)) \text{ but not both} : t(x), t(x')\in \mathcal{T} \text{ with } t(x) \ne t(x') \}
\]
that contains precisely one active, directed path for each pair of terminal states  for which  
\begin{align}\label{eq:sum_condition}
\sum_{\substack{(x,x')\in \mathbb{S}\times \mathbb{S}\\
    x\ne x'}} |\gamma(t(x),t(x'))| \left( \frac{\pi(x)\pi(x')}{\pi_{\text{min}}\big(\gamma(t(x),t(x'))\big)}\right) < \infty,
\end{align}
where $|\gamma(t(x),t(x'))| \left( \frac{\pi(x)\pi(x')}{\pi_{\text{min}}\big(\gamma(t(x),t(x'))\big)}\right)$ is taken to be zero if $\gamma(t(x),t(x')) \notin \Gamma_{\mathcal{T}}$.

\item (Uniform lower bound for transition rates.) There is a minimum for the transition rates along all the active paths described in parts 1 and 2 of this assumption (including the paths $\{\gamma_x\}$ and those in $\Gamma_{\mathcal{T}}$). That is, there exists a $c_{\text{min}}>0$ for which both the following hold,
\begin{align}
    \label{eq:activepathcond}
    \begin{split}
    \inf_{x\in \mathbb{S}} \inf_{(z,z') \text{ is an edge in $\gamma_x$}} \{q(z,z')\} &\ge c_{\text{min}}\\
    \inf_{\gamma \in \Gamma_\mathcal{T}} \inf_{(z,z') \text{ is an edge in $\gamma$}} \{q(z,z')\} &\ge c_{\text{min}}.
    \end{split}
\end{align}
\hfill $\triangle$
\end{enumerate}
\end{assumption}

\begin{rem}\label{rem:main assumption}
Note that condition \eqref{eq:no infinitely many paths intersects} simply states that there is a bound on the number of paths of the form $\gamma_z = \gamma(z,t(z))$ that a particular edge can be contained within.  Moreover, \eqref{eq:no infinitely many paths intersects} holds if \eqref{eq:overlineL} holds and if the sizes of the possible  transitions for $X$ are uniformly bounded. Also note that \eqref{eq:activepathcond} holds if the rates $q_x=\sum_{z\in \mathbb S\setminus x} q(x,z)$ satisfy $\lim_{|x|\to \infty}q_x=\infty$. \hfill $\dagger$
\end{rem}

We  now state the main technical result developed in this paper.  Loosely, it says that if Assumption \ref{assump:main} holds, then the spectral gap is positive and the model is exponentially ergodic.

\begin{thm}\label{thm:main for poincare}

Suppose $X$ is an irreducible and positive-recurrent continuous-time Markov chain with countable state space $\S\subset \Z^d$ and stationary distribution $\pi$ for which Assumption \ref{intcnd} holds.   Suppose that Assumption \ref{assump:main} also holds.  
Then there is a $C>0$ for which
\begin{align}\label{eq:toshow111}
    \EE(f,f) \ge C\cdot  Var_\pi(f)
\end{align}
for all $f\in Dom(\mathcal{A})$ with $\|f\|_\infty < \infty$, and so $\gap(\EE) \ge C$.  Moreover, we have the following two conclusions.
\begin{enumerate}[(i)]
\item \eqref{eq:toprove} holds with $B(x) = \dfrac{2}{\pi(x)}$ and $\eta = 2C$.  That is, 
%there is an $\eta >0$ for which
\begin{align*}
    \|P^t(x,\cdot) - \pi(\cdot)\|_{\text{TV}} \le \frac{2}{\pi(x)} e^{-2C t},
\end{align*}
for all $t \ge 0$.

\item For a fixed $\varepsilon>0$ and each $x \in \mathbb{S}$  the mixing time defined in \eqref{eq:mixing time} satisfies 
\[
    \tau_\varepsilon^x \le \frac{1}{2C}
\left(\left|\ln{(\pi(x))}\right|+\left|\ln \left(\frac{\varepsilon}{2}\right)\right|\right) = O(|\ln(\pi(x))|).
\]
\end{enumerate}
\end{thm}

\begin{proof}
Note that conclusions (i) and (ii) follow directly from a combination of \eqref{eq:toshow111}, Theorem \ref{thm:adlkjfkd}, and Lemma \ref{lem:randome11}. Hence, we only show \eqref{eq:toshow111}.  

Let $f\in Dom(\mathcal{A})$ with $\|f\|_\infty<\infty$.  Due to   Lemma \ref{lem:eff}, which states that $\EE(f,f) = \EE^*(f)$ for the $f$ under consideration, we need to show that
\begin{equation}\label{PoinIneq}
    Var_\pi (f) \le \frac{1}{C}\cdot  \mathcal E^*(f)
\end{equation}
for some $C>0$ that does not depend upon the choice of $f$.

The proof of \eqref{PoinIneq} is organized into three steps, and should be compared to the analysis of Example \ref{ex:motivation}.  
\vspace{0.3cm}

\noindent
\textbf{Step 1: Path decomposition. }  By definition, we have that 
\begin{align}\label{eq:9879876}
    Var_\pi(f) = \sum_{x\in \mathbb{S}}\left(f(x) - \sum_{x'\in \mathbb{S}} f(x')\pi(x')\right)^2\pi(x).
\end{align}
By Jensen's inequality we have
\[
\left( f(x)- \sum_{x' \in \mathbb{S}}f(x')\pi(x') \right)^2=\left( \sum_{x' \in \mathbb{S}}(f(x)-f(x'))\,\pi(x') \right)^2 \leq 
\sum_{x'\in \mathbb{S}} (f(x)-f(x'))^2\pi(x').
\]
Putting this back into \eqref{eq:9879876} yields
\begin{align}
Var_\pi(f) &= \sum_{x\in \mathbb{S}} \left( f(x)- \sum_{x' \in \mathbb{S}}f(x')\pi(x') \right)^2\pi(x) \le\sum_{x \in \mathbb{S}} \sum_{x' \in \mathbb{S}} (f(x)-f(x'))^2\pi(x)\pi(x').\label{eq:var}
\end{align}

We will now consider the term $(f(x)-f(x'))^2$ for $x \ne x'$.   For a given pair $x,x'\in \mathbb{S}$, we enumerate the paths $\gamma_x$, $\gamma_{x'}$ and $\gamma(t(x),t(x'))\in \Gamma_{\mathcal{T}}$ in Assumption \ref{assump:main} as 
\begin{align*}
\gamma_x&=(x_1,x_2,\dots,x_{|\gamma_x|}), \\
\gamma_{x'} &=(x'_1, x'_2,\dots,x'_{|\gamma_{x'}|})\\
\gamma(t(x),t(x'))&=(z_1,z_2,\dots,z_{|\gamma(t(x),t(x'))|}),
\end{align*}
where we note that $z_1 = t(x) = x_{|\gamma_x|}$ and $z_{|\gamma(t(x),t(x'))|} = t(x') = x'_{|\gamma_{x'}|}$.
Note that we are, without loss of generality, specifically (and arbitrarily) assuming that the active path in $\Gamma_{\mathcal{T}}$ connecting $t(x)$ and $t(x')$ is $\gamma(t(x),t(x'))$ as opposed to $\gamma(t(x'),t(x))$. We now split the term $f(x)-f(x')$ in \eqref{eq:var}  into three sums via a telescoping sum using the three active paths $\gamma_x$, $\gamma(t(x),t(x'))$, and $\gamma_{x'}$.  
We obtain
\begin{align}
(f(x)-f(x'))^2 &=\Bigg (-\sum_{i=1}^{|\gamma_x|-1}(f(x_{i+1})-f(x_{i})) - \sum_{i=1}^{|\gamma(t(x),t(x'))|-1}(f(z_{i+1})-f(z_{i}))\notag\\
&\hspace{.3in}+\sum_{i=1}^{|\gamma_{x'}|-1}(f(x'_{i+1})-f(x'_{i})) \Bigg ) ^2\notag\\
&\le 2\left (- \sum_{i=1}^{|\gamma_x|-1}(f(x_{i+1})-f(x_{i})) + \sum_{i=1}^{|\gamma_{x'}|-1}(f(x'_{i+1})-f(x'_{i})) \right  ) ^2 \notag\\
&\hspace{0.1in}+ 2\left (\sum_{i=1}^{|\gamma(t(x),t(x'))|-1}(f(z_{i+1})-f(z_{i})) \right )^2\notag\\
\begin{split}\label{eq:path from x to x'}
&\le 4 \overline L \left ( \sum_{i=1}^{|\gamma_x|-1}(f(x_{i+1})-f(x_{i}))^2 + \sum_{i=1}^{|\gamma_{x'}|-1}(f(x'_{i+1})-f(x'_{i}))^2 \right )\\
&\hspace{0.1in}+ 2|\gamma(t(x),t(x'))|\sum_{i=1}^{|\gamma(t(x),t(x'))|-1}(f(z_{i+1})-f(z_{i}))^2,
\end{split}
\end{align}
where we used $(\sum_{i=1}^n a_i)^2 \le n\sum_{i=1}^n a_i^2$ for the two inequalities above, and we remind that $\overline L$ is defined in \eqref{eq:overlineL}.
We may plug the inequality in \eqref{eq:path from x to x'} back into the inequality \eqref{eq:var} to obtain the following
\begin{align}
Var_\pi(f) \le&\sum_{(x,x')\in \mathbb S\times \mathbb S}(f(x)-f(x'))^2\pi(x)\pi(x')\notag\\
\le&  4 \overline L \sum_{\substack{(x,x')\in \mathbb S\times \mathbb S\\
x\ne x'}}  \left ( \sum_{i=1}^{|\gamma_x|-1}(f(x_{i+1})-f(x_{i}))^2 + \sum_{i=1}^{|\gamma_{x'}|-1}(f(x'_{i+1})-f(x'_{i}))^2 \right ) \pi(x)\pi(x')   \label{E1}\\
& +2\sum_{\substack{(x,x')\in \mathbb S\times \mathbb S\\
x\ne x'}}\left ( |\gamma(t(x),t(x'))|\sum_{i=1}^{|\gamma(t(x),t(x'))|-1}(f(z_{i+1})-f(z_{i}))^2 \right ) \pi(x)\pi(x'). \label{E2}
\end{align}

The summations (\ref{E1}) and (\ref{E2}) correspond to the ends and the middle parts of the path from $x$ to $x'$, respectively. We must now show that each is upper bounded by   $c\cdot \EE^*(f)$, for some positive constant $c$.  We handle the sums individually in the next two steps of the proof.

\vspace{0.3cm}
\noindent
\textbf{Step 2: Upper bound for (\ref{E2}). }
For each pair $(x,x')\in \mathbb S \times \mathbb S$ with $x\ne x'$, we have 
\begin{align}
    &\sum_{i=1}^{|\gamma(t(x),t(x'))|-1}(f(z_{i+1})-f(z_{i}))^2 \pi(x)\pi(x')\notag\\
    &=\sum_{i=1}^{|\gamma(t(x),t(x'))|-1}(f(z_{i+1})-f(z_{i}))^2\, \frac{q({z_{i},z_{i+1})}\,\pi(z_{i})}{q({z_{i},z_{i+1})}\,\pi(z_{i})}\pi(x)\pi(x') \notag \\
    &\le \frac{1}{c_{\text{min}}} \,\left(\max_{z\in \gamma(t(x),t(x'))}\frac{\pi(x)\pi(x')}{\pi(z)} \right) \sum_{i=1}^{|\gamma(t(x),t(x'))|-1}(f(z_{i+1})-f(z_{i}))^2\,q({z_{i},z_{i+1})} \,\pi(z_{i})\notag\\
    &\le 2\frac{1}{c_{\text{min}}}  \left(\frac{\pi(x)\pi(x')}{\pi_{\text{min}}(\gamma(t(x),t(x')))
    %\min_{z\in \gamma(t(x),t(x'))} \pi(z)
    }\right)\EE^*(f), \label{eq:bounded by diri}
 \end{align}
 where $c_{\text{min}}$ is as in \eqref{eq:activepathcond} (and is used to bound the term $1/q(z_i,z_{i+1})$), and the term $2\EE^*(f)$ arises by restricting the sum in \eqref{eq:78657857} from all of $\mathbb{S}\times \mathbb{S}$ to the particular path $\gamma(t(x),t(x'))$.  Note that  $\gamma(t(x),t(x'))$ being an active path plays a key role above since we require that $q(z_i,z_{i+1})>0$.
It follows from  \eqref{eq:bounded by diri} that the summation in (\ref{E2}) is bounded in the following manner
\begin{align}
2\sum_{\substack{(x,x')\in \mathbb S\times \mathbb S\\
x\ne x'}}&\left ( |\gamma(t(x),t(x'))|\sum_{i=1}^{|\gamma(t(x),t(x'))|-1}(f(z_{i+1})-f(z_{i}))^2 \right ) \pi(x)\pi(x') \notag \\
&\le 4\frac1{c_{\text{min}}}\, \EE^*(f) \sum_{\substack{(x,x')\in \mathbb S\times \mathbb S\\ 
x\ne x'}}|\gamma(t(x),t(x'))| \left(\frac{\pi(x)\pi(x')}{\pi_{\text{min}}(\gamma(t(x),t(x')))
    %\min_{z\in \gamma(t(x),t(x'))} \pi(z)
    }\right)
 \label{eq:sum for E2}\\
 &\le c_1 \cdot \EE^*(f),
\end{align}
for some constant $c_1 >0$,  where we utilized   \eqref{eq:sum_condition} in Assumption \ref{assump:main} for the final inequality.

\vspace{0.3cm}
\noindent
\textbf{Step 3: Upper bound for (\ref{E1}). }
By the symmetry of \eqref{E1}, it is enough to show that there exists some $c>0$ for which
\begin{align}\label{eq:target for E1 part}
\sum_{x\in \mathbb S}\sum_{i=1}^{|\gamma_x|-1}(f(x_{i+1})-f(x_{i}))^2  \pi(x) <c\cdot \EE^*(f).
\end{align}
We have
\begin{align}
&\sum_{x\in \mathbb S}\sum_{i=1}^{|\gamma_x|-1}(f(x_{i+1})-f(x_{i}))^2  \pi(x) \notag\\
&= \sum_{x\in \mathbb S}\sum_{i=1}^{|\gamma_x|-1}(f(x_{i+1})-f(x_{i}))^2\frac{q{(x_{i},x_{i+1})}\,\pi(x_{i})}{q{(x_{i},x_{i+1})}\,\pi(x_{i})}\pi(x) \notag \\
&\le \frac{1}{c_{\text{min}}}\,\left ( \sup_{x\in \mathbb S} \frac{\pi(x)}{\pi_{\text{min}}(\gamma(x))} \right ) \,\sum_{x\in \mathbb S}\sum_{i=1}^{|\gamma_x|-1}(f(x_{i+1})-f(x_{i}))^2 \,q{(x_{i},x_{i+1})}\,\pi(x_{i})  \label{eq:sum near the edge}.
\end{align}
However, the term in parentheses is bounded by \eqref{eq:supmaxassump} in Assumption \ref{assump:main}. For the double sum  in \eqref{eq:sum near the edge}, by considering how many times each particular pair $(x,z) \in \mathbb{S}\times \mathbb{S}$ could be an edge, we  have
\begin{align*}
    \sum_{x\in \mathbb S}\sum_{i=1}^{|\gamma_x|-1}(f(x_{i+1})-f(x_{i}))^2 \,q{(x_{i},x_{i+1})}\,\pi(x_{i}) &\le \overline M \sum_{x,z\in \mathbb{S}} (f(z)-f(x))^2q(x,z) \pi(x)\\
    &= 2\overline M \EE^*(f).
\end{align*}
where $\overline M$ is as in \eqref{eq:no infinitely many paths intersects}. Hence, the result is shown.
\end{proof}

\subsection{Special cases that are easier to apply for models on \texorpdfstring{$\Z^d_{\ge 0}$}{Lg}}
\label{sec:corrolaries}
The conditions in Assumption \ref{assump:main} are quite technical.  We therefore present a few corollaries that are significantly more ``user-friendly.'' 

Our first corollary considers models restricted to $\Z^d_{\ge 0}$ and states that the conditions in Assumption \ref{assump:main} hold so long as (i) the stationary distribution has a type of super-exponential decay, and (ii)  transitions of the form $\pm e_i$, where $\{e_i\}$ is the canonical basis for $\Z^d$, have positive rate so long as the transition would not push the system into a region with a negative component. 
We note that the super-exponential decay condition of \eqref{eq:super expo} below  holds for many probability distributions showing up in our  area of interest of biochemical reaction systems (see section \ref{sec:modelsandresults}).  Examples include models whose stationary distribution is a product  of Poissons \cite{AndProdForm}, some products of Poisson and Beta-binomials \cite{bibbona2020stationary}, and some models with non-mass-action kinetics \cite{anderson2016product, hong2021derivation}.  Moreover, in \cite{XuHansenWiuf2022}, a detailed analysis of the decay of the tails of the stationary distribution for one-dimensional models is given; even though without a lower bound on $\pi(x)$ we cannot deduce \eqref{eq:super expo}, it seems possible the link will be established with future work.

Note that in Corollary \ref{cor:main} below the fact that $\S = \Z^d_{\ge 0}$ is irreducible follows from condition 2 in that corollary. 
\begin{cor}\label{cor:main}
    Suppose $X$ is a positive-recurrent continuous-time Markov chain with irreducible state space $\S = \Z^d_{\ge 0}$ and stationary distribution $\pi$.  Suppose that  Assumptions \ref{intcnd} hold. Suppose further that the following two conditions hold.
    \begin{enumerate}
        \item There exist an integer $K>0$ and an $\alpha>0$ so that for each $i \in \{1,\dots, d\}$,
    \begin{align}\label{eq:super expo}
    \dfrac{\pi(x)}{\pi(x-e_i)}\le \dfrac{1}{x_i^\alpha}    
    \end{align}    
    if $x$ is such that $x_i \ge K$, and
    \item There is a $c_{\text{min}}>0$ so that for each $i\in \{1,\dots, d\}$, we have 
    \[
    \displaystyle \inf_{x\in\mathbb Z^d_{\ge 0}} q(x,x+e_i)\ge c_{\text{min}}  \quad \text{and} \quad \inf_{\substack{x \in \mathbb Z^d_{\ge 0}\\ x_i \ge 1}}q(x,x-e_i) \ge c_{\text{min}}.
    \]
    \end{enumerate}
    Then the conditions in Assumption \ref{assump:main} hold. Hence, the conclusions of Theorem \ref{thm:main for poincare} hold as well.
\end{cor}

We begin with the following lemma, which states that a process satisfying the conditions of Corollary \ref{cor:main} has all moments. 

\begin{lem}\label{lem: moment of super exponential}
   Let $X$ and $\pi$ be as in Corollary \ref{cor:main}.  Let $m \ge 3$ be an integer.  Then there is a $c_m>0$, depending upon $m$, for which
   \[
    \pi(x) \le c_m \prod_{i=1}^d \frac{1}{(x_i + 1)^m},
   \]
   for all $x\in \Z^d_{\ge 0}$.
  \end{lem}

  \begin{proof}
  Let $K$ and $\alpha$ be as in the statement of Corollary \ref{cor:main}.  Fix a $K'>K$ so that 
  \begin{equation}\label{eq:bound8796896}
   \frac{(K!)}{(x!)^\alpha} \le \frac{1}{x^m},
  \end{equation}
  for all $x \ge K'$.
  With this choice, if $x_1 \ge K'$, then \eqref{eq:super expo} with $i =1$  implies that 
    \begin{align}
       \pi(x)&=\frac{\pi(x)}{\pi(x-e_1)}\frac{\pi(x-e_1)}{\pi(x-2e_1)}\cdots \frac{\pi(x-(x_1-K-1)e_1)}{\pi(x-(x_1-K)e_1)}\pi(x-(x_1-K)e_1) \notag \\
       &\le  \frac{1}{x_1^\alpha} \cdot \frac{1}{(x_1-1)^\alpha} \cdots \frac{1}{(K+1)^\alpha} \pi(x-(x_1 -K)e_1)\tag{repeatedly applying \eqref{eq:super expo}}\\
       &= \frac{(K!)^\alpha }{(x_1!)^\alpha} \pi(x-(x_1 -K)e_1)\notag\\
       & \le \frac{1}{x_1^{m}}\cdot \pi(x-(x_1-K)e_1),\label{eq:needed}
    \end{align}
    where the final inequality follows from \eqref{eq:bound8796896}.
         To generalize the above, let $I_{K'}(x)=\{i : x_i \ge K'\}$ be those indices for which the  $ith$ coordinate of $x$ is larger than $K'$.  For each $x$ we define $u(x)\in \Z^d_{\ge 0}$ via
    \begin{align*}
        u(x)_i=\begin{cases}
        K \quad &\text{if $i \in I_{K'}(x)$}\\
        x_i &\text{otherwise.}
        \end{cases}
    \end{align*}
    Applying the above inequality to those indices $I_{K'}(x)$, there is a   $c_m>0$  such that for any $x \in \Z^d_{\ge 0}$ 
    \begin{align*}
       \pi(x) &\le \prod_{i\in I_{K'}(x)} \frac{1}{x_i^m}\cdot \pi\left (x-\sum_{i\in I_K(x)}(x_i-K)e_i \right ) = \prod_{i\in I_{K'}(x)} \frac{1}{x_i^m}\pi(u(x)) \nonumber \\
       &\le c_m\prod_{i=1}^d \frac{1}{(x_i+1)^m} \pi(u(x)) \le c_m\prod_{i=1}^d \frac{1}{(x_i+1)^m},       
    \end{align*}    
    where we used that $\pi(u(x)) \le 1$ in the final inequality.  Hence, the result is shown.
\end{proof}

We turn to the proof of Corollary \ref{cor:main}.

\begin{proof}[Proof of Corollary \ref{cor:main}]

We begin by defining $t(x)$, the terminal state, for each $x \in \Z^d_{\ge 0}$.  To do so, we begin by letting $k_0 \ge K+\frac3{\alpha} + 1$ be a positive integer 
and let $I_{k_0}(x) = \{i : x_i \ge k_0\}$ be those indices for which $x_i \ge k_0$.  We then define the terminal state for $x$ component-wise via
\begin{equation}
t(x)_i = \begin{cases}
x_i - \lceil \frac{3}{\alpha}\rceil & \text{ if } i \in I_{k_0}(x)\\
x_i & \text{otherwise},
\end{cases}
\end{equation}
where $\lceil \frac{3}{\alpha}\rceil$ is the smallest integer larger than $\frac{3}{\alpha}$ (i.e., we round up).

  We turn to defining $\gamma_x=\gamma(x,t(x))$ for each $x$.  If $\|x\|_\infty < k_0$, in which case $I_{k_0}(x)$ is empty, then $t(x) = x$ and we take $\gamma(x,t(x)) = (x)$.  Now suppose that  $\|x\|_{\infty} \ge k_0$.  In this case, $I_{k_0}(x)$ is non-empty and for each $i\in I_{k_0}(x)$ we will subtract off $e_i$ exactly $\lceil \frac{3}{\alpha}\rceil$ times.  Note that there are $\displaystyle \frac{(|I_{k_0}(x)|\cdot \lceil \frac{3}{\alpha}\rceil)!}{(\lceil \frac{3}{\alpha}\rceil !)^{|I_{k_0}(x)|}}$ possible such paths connecting $x$ and $t(x)$ using only subtractions of $e_i$ for $i \in I_{k_0}(x)$.  We choose one of those paths at random.  

For example, consider the situation with $d =3$, $\alpha=1$, 
and $x = (k_0 + 2, k_0 - 1, k_0 + 1)$.  Here, $I_{k_0}(x)=\{1,3\}$ and $\lceil \frac{3}{\alpha}\rceil = 3$ and so $t(x) = (k_0-1,k_0-1,k_0 - 2)$.  There are $\frac{(2\cdot 3)!}{(3!)^2} = 20$ possible paths to choose from. If  we choose the path that decreases the first component 3 times in a row followed by the third component 3 times in a row, then the specific path  $\gamma_x = \gamma(x,t(x))$ is 
\begin{equation*}
    ((k_0 + 2, k_0 - 1, k_0 + 1), \dots , (k_0 - 1, k_0 - 1, k_0 + 1),\dots, (k_0-1, k_0 - 1,k_0-2)).
\end{equation*}

Returning to the general case, we now observe the following.
\begin{itemize}
    \item The paths constructed have a maximum length of $\lceil \frac{3}{\alpha}\rceil d$, so \eqref{eq:overlineL} holds.
    
    \item Each edge in $\gamma(x,t(x))$ must be within $\lceil \frac{3}{\alpha}\rceil d$ transitions  taken from the set $\{-e_i: i \in \{1,\dots, d\}\}$.  Hence, the number of such paths  that a particular edge can be contained within is bounded and \eqref{eq:no infinitely many paths intersects} holds. 
    
    \item Combining \eqref{eq:super expo} with the construction above immediately yields 
    \[
    \sup_{x\in \Z^d_{\ge 0}} \frac{\pi(x)}{\pi_{\text{min}}(\gamma_x)}
    %\min_{z\in \gamma_x}  \pi(z)} 
    \le 1,
    \]
    and so   \eqref{eq:supmaxassump} holds.  
\end{itemize}
Hence, condition 1 of Assumption \ref{assump:main} holds.

To verify condition 2 of Assumption \ref{assump:main}, we first
point out a bound we will use later.  Doing an analysis similar to that found in \eqref{eq:needed}, we may apply \eqref{eq:super expo} one time for each edge in the path $\gamma(x,t(x))$ to conclude that there exist constants $c_1>0$ and $c_2>0$ 
that are independent of the choice of  $x \in \Z^d_{\ge 0}$, such that
\begin{align}\label{eq:x and t(x)}
    \frac{\pi(x)}{\pi(t(x))} &\le  \prod_{i \in I_{k_0}(x)} \left(\frac{c_1}{(x_i)^{\alpha}}\right)^{\lceil \frac{3}{\alpha}\rceil} \le \prod_{i= 1}^d \frac{c_2}{(x_i+1)^3}.
\end{align}

We turn to the paths connecting the terminal states.
  Recall that $\mathcal{T}$ is the set of terminal states $\{t(x): x \in \Z^d_{\ge 0}\}$.
 We must still define $\Gamma_{\mathcal{T}}$ so that \eqref{eq:sum_condition} holds.
Towards that end, select a pair of terminal states: $t(x),t(x') \in \mathcal{T}$, with $t(x) \ne t(x')$.    
We arbitrarily choose to orient the path as starting from $t(x)$ and terminating at $t(x')$.  That is, we construct the path $\gamma(t(x),t(x'))$.
Moreover, to construct the path $\gamma(t(x),t(x'))$ we only use transitions of the form $\{\pm e_i, i \in \{1,\dots d\}\}$, where, as always, $\{e_i\}$ is the canonical basis in $\Z^d$.  

 For two vectors $x,y\in\Z^d$ we denote by $x\wedge y$ the vector whose $i$th coordinate is $x_i \wedge y_i = \min\{x_i,y_i\}$.
 The main idea for the construction of the desired $\gamma(t(x),t(x'))$ is to use two sub-paths: one connects $t(x)$ to $t(x)\wedge t(x')$ by only using transitions of the form $\{-e_i, i \in \{1,\dots, d\}\}$, and the other one connects $t(x)\wedge t(x')$ to $t(x')$ by only using transitions of the form $\{e_i, i \in \{1,\dots, d\}\}$.
 
 We now make this precise.   Let $I^-=\{ i : t(x)_i > t(x')_i\}$ and $I^+=\{ i : t(x)_i < t(x')_i\}$.  Note that at least one of $I^-$ or $I^+$ is non-empty since $t(x)\ne t(x')$.  Note that there are 
 \[
  \frac{\left(\sum_{i\in I^-} (t(x)_i - t(x')_i)\right)!}{\prod_{i\in I^-} (t(x)_i-t(x')_i)!}
 \]
 paths connecting $t(x)$ to $t(x)\wedge t(x')$ using only transitions of the form $-e_i$ for those $i\in I^-$.  The particular choice of path does not matter and we select one at random.  We now have our path, $\gamma^- = \gamma(t(x),t(x)\wedge t(x'))$ connecting $t(x)$ to $t(x) \wedge t(x')$.  
 
 We then similarly connect $t(x)\wedge t(x')$ to $t(x')$ via a path using only transitions of the form $e_i$ for $i \in I^+$.  We denote this path via $\gamma^+ = \gamma(t(x)\wedge t(x'), t(x'))$.  The path connecting $t(x)$ to $t(x')$, denoted $\gamma(t(x),t(x'))$, is then defined by moving first along $\gamma^-$ and then along $\gamma^+.$
 
 Note that, by construction, if $z$ is an element of $\gamma(t(x),t(x'))$, i.e. a state visited along the path $\gamma^-$ or $\gamma^+$, then we either have $z \le t(x)$ (meaning $z_i \le t(x)_i$ for all $i\in \{1,\dots, d\}$) or $z \le t(x')$.  Combining this fact with \eqref{eq:super expo} implies that 
\begin{align}\label{eq:min pi z}
    \frac{1}{\displaystyle \pi_{\text{min}}(\gamma(t(x),t(x')))}
    %\min_{z\in\gamma(t(x),t(x'))}\pi(z)} 
    \le  C_K \max \left \{\frac{1}{\pi(t(x))}, \frac{1}{\pi(t(x'))} \right \},
\end{align}
where the constant $C_K>0$ is needed since some coordinates of $t(x)$ or $t(x')$ may be smaller than $K$.

Finally,  note  that 
 \begin{equation}\label{eq:bound1111}
    |\gamma(t(x),t(x')| =  \sum_{i \in I^+ \cup I^-} |t(x)_i - t(x')_i| \le |x| + |x'|.
 \end{equation}

We may combine all of the above to show that \eqref{eq:sum_condition} holds:  
\begin{align}
    &\sum_{\substack{(x,x')\in \Z^d_{\ge 0}\times \Z^d_{\ge 0}\\
    x\ne x'}} |\gamma(t(x),t(x'))| \left( \frac{\pi(x)\pi(x')}{\pi_{\text{min}}(\gamma(t(x),t(x')))}\right) \notag\\
    %\min_{z\in \gamma(t(x),t(x'))} \pi(z)}\right) \notag\\
    &\le C_K \sum_{\substack{(x,x')\in \Z^d_{\ge 0}\times \Z^d_{\ge 0}\\
    x\ne x'}} (|x|+|x'|) \pi(x)\pi(x')\max \left \{\frac{1}{\pi(t(x))}, \frac{1}{\pi(t(x'))} \right \}\tag{using \eqref{eq:min pi z} and \eqref{eq:bound1111}} \\
     &\le C_K \sum_{\substack{(x,x')\in \Z^d_{\ge 0}\times \Z^d_{\ge 0}\\
    x\ne x'}} (|x|+|x'|) \pi(x)\pi(x')\left (\frac{1}{\pi(t(x))} + \frac{1}{\pi(t(x'))} \right )\notag\\
     &\le C_K \cdot c_3 \sum_{\substack{(x,x')\in \Z^d_{\ge 0}\times \Z^d_{\ge 0}\\
    x\ne x'}} (|x|+|x'|) \left (\pi(x')\prod_{i=1}^d\frac{1}{(x_i+1)^3} + \pi(x)\prod_{i=1}^d\frac{1}{(x'_i+1)^3} \right )\tag{from \eqref{eq:x and t(x)}},
\end{align}
which is finite by applying Lemma \ref{lem: moment of super exponential} with $m = 3$.
Hence, we have demonstrated that condition 2 of Assumption \ref{assump:main} holds.

Finally, since all the paths we constructed above only consist of birth and death events (i.e. the transitions given by $e_i$ and $-e_i$), the second condition of Corollary \ref{cor:main} (pertaining to the infimum of the rates)  implies that condition 3 of Assumption \ref{assump:main} holds.  Hence, the proof is complete.
\end{proof}

If the stationary distribution $\pi$ is of a product form, while also satisfying condition 1 of Corollary \ref{cor:main}, then we can relax the second condition in Corollary \ref{cor:main}.  That condition required that births and deaths can happen for each index at each state in $\Z^d_{\ge 0}$ (so long as the transition did not push the process out of $\Z^d_{\ge 0}$), and moreover that the rates for those transitions had a uniform lower bound. 
Loosely speaking, the 
new condition is that there is a partition of the indices $\{1,\dots, d\}$, denoted by the disjoint, non-empty sets $\{J_0,\dots, J_m\}$ so that (i) the transitions $\pm e_i$ have positive rate for all states if $i \in J_0$ (so long as the transition does not push the process negative in the $i$th component), and (ii) the transitions $\pm e_i$, for $i \in J_\ell$, have positive rate for any state for which $x_j$ is greater than some given threshold (denoted $N$ below) for some $j \in \cup_{i = 0}^{\ell-1}J_i$.  Moreover, the rates of these transitions must still have a uniform lower bound. Thus, loosely, the components associated to the  indices in $J_0\cup \cdots \cup J_{\ell-1}$ act as catalysts, or at least switches, for the birth and death transitions for the components associated to the indices in $J_\ell$.  Note that  if $J_0 = \{1,\dots, d\}$ then this condition is the same as condition 2 of Corollary \ref{cor:main}.

Note that in the corollary below the fact that $\S = \Z^d_{\ge 0}$ is irreducible follows from conditions 2 and 3. 
\begin{cor}\label{cor:relaxed}
    Suppose $X$ is an irreducible and positive-recurrent continuous-time Markov chain with countable state space $\S = \Z^d_{\ge 0}$ and stationary distribution $\pi$ for which Assumption \ref{intcnd}  holds.
    Suppose that the following three conditions hold.
    \begin{enumerate}
        \item $\pi(x)=\prod_i^d \pi_i(x_i)$ where $\pi_i$ is a probability measure on $\mathbb Z_{\ge 0}$.  Moreover, there is an integer $K\ge 1$ and a constant $\alpha>0$ so that for each $i\in \{1,\dots,d\}$ we have  $\dfrac{\pi_i(n)}{\pi_i(n-1)} \le \dfrac{1}{n^\alpha}$ whenever $n \ge K$.  
        
        \item There exists an $c_{\text{min}} >0$ and a non-empty subset $J_0 \subset \{1,2,\dots,d\}$ so that for each $i\in J_0$ we have 
        \begin{align*}
         q(x,x+e_i)&\ge c_{\text{min}} 
        \quad \text{for all $x\in \Z^d_{\ge 0}$},\\
        q(x,x-e_i) &\ge c_{\text{min}} \quad \text{for all $x\in \Z^d_{\ge 0}$ with $x_i \ge 1$}
        \end{align*}

        \item There exist disjoint subsets of $\{1,\dots,d\}$, denoted $J_1,\dots, J_m$, so that $\{J_0,J_1, \dots, J_m\}$ is a partition of $\{1,\dots, d\}$, and there exists an integer $N\ge 1$  so that for each $i \in J_\ell$ with $\ell \ge 1$, we have
        \begin{enumerate}
            \item[(i)] $q(x,x+e_i)\ge c_{\text{min}}$ when $x_j\ge N$ for some $j\in J_0 \cup \cdots \cup J_{\ell-1}$, and 
            \item[(ii)] $q(x,x-e_i)\ge c_{\text{min}}$ if $x_i\ge 1$ and $x_j\ge N$ for some $j\in J_0 \cup \cdots \cup J_{\ell-1}$.
            \end{enumerate}
    \end{enumerate}
Then the conditions in Assumption \ref{assump:main} hold. Hence, the conclusions of Theorem \ref{thm:main for poincare} hold as well.
\end{cor} 

\begin{proof}
Let $N_{K,\alpha}=N+K+\lceil\frac{3}{\alpha} \rceil+1$.
First, we define the terminal point $t(x)$ for each $x$ and construct $\gamma_x := \gamma(x,t(x))$.  
Let 
\begin{align*}
    D:=\{ x \in \mathbb Z^d_{\ge 0} : x_i \ge N_{K,\alpha} \text{ for each $i$}\}.
\end{align*}
If $x\in D$, then   we define $t(x)$ component-wise via  
\[
    t(x)_i=x_i-\left \lceil \frac{3}{\alpha} \right\rceil \text{ for each $i$}.
\]
As in the proof of Corollary \ref{cor:main}, for this case we define $\gamma_x = \gamma(x,t(x))$ to be one of the active  paths that can be constructed from $x$ to $t(x)$ that only utilizes transitions of the form $-e_i$, where, as always, $\{e_i\}$ is the canonical basis of $\Z^d$.

   Suppose now that $x\in D^c$. In this case, we have $x_i < N_{K,\alpha}$ for at least one $i \in \{1,\dots,d\}$. We define two  points: an intermediate point, $t_{\text{inter}}(x),$ and then finally the terminal point $t(x)$.  First, we define $t_{\text{inter}}(x)$ componentwise via
   \begin{align*}
       (t_{\text{inter}}(x))_i =\begin{cases}
            N_{K,\alpha} \quad &\text{if $x_i < N_{K,\alpha}$}\\
           x_i &\text{if $x_i \ge N_{K,\alpha}$}.
       \end{cases}
   \end{align*}
   Next, we define $t(x)$, whose $i$th coordinate is
   \begin{align*}
    t(x)_i=(t_{\text{inter}}(x))_i-\left\lceil \frac{3}{\alpha} \right\rceil.
    \end{align*}

    We now construct a path, $\gamma_x^1 = \gamma(x,t_{\text{inter}}(x))$ connecting $x$ to $t_{\text{inter}}(x)$ in the following manner.
    \begin{enumerate}
        \item For $i \in J_0$ with $x_i < N_{K,\alpha}$, use the transitions $+e_i$ exactly $N_{K,\alpha}-x_i$ times.  Repeat for each such $i \in J_0$, with the ordering of the elements of $J_0$ not important, and any will do.

        \item Repeat the above for $J_\ell$, $\ell = 1, \dots, m$, taking the $\ell$ in sequential order ($\ell = 1$ first, $\ell =2$ second, etc.). The ordering of the elements of each $J_\ell$ is not important, and any will do.
    \end{enumerate}
     Next, we now construct a path, $\gamma_x^2 = \gamma(t_{\text{inter}}(x), t(x))$ connecting  $t_{\text{inter}}(x)$ to $t(x)$ in the following manner. 

    \begin{enumerate}
        \item For $i \in J_0$, use the transitions $-e_i$ exactly $\lceil \frac{3}{\alpha}\rceil$ times.  Repeat for each  $i \in J_0$, with the ordering of the elements of $J_0$ not important, and any will do.

        \item Repeat the above for $J_\ell$, $\ell = 1, \dots, m$, taking the $\ell$ in sequential order ($\ell = 1$ first, $\ell =2$ second, etc.). The ordering of the elements of each $J_\ell$ is not important, and any will do.
    \end{enumerate}
    Finally, we take $\gamma_x=\gamma(x,t(x))$ to be the path that first moves along $\gamma_x^1$ and then along $\gamma_x^2$.

For example, consider the situation with $d =3$, and $J_0=\{1\}, J_1=\{2\}$, and $J_3=\{3\}$. We also assume that $\alpha=1$ so that $\lceil \frac{3}{\alpha}\rceil = 3$. Then for 
$x = (N_{K,\alpha} - 2, N_{K,\alpha} - 1, N_{K,\alpha} + 1) \in D^c$, 
we have 
\[
t_{\text{inter}}(x) =(N_{K,\alpha}, N_{K,\alpha}, N_{K,\alpha} + 1) \quad \text{ and } \quad t(x)=(N_{K,\alpha}-3,N_{K,\alpha}-3,N_{K,\alpha}-2).
\]
In this case, $\gamma^1_x$ is 
\begin{align*}
    \big(&(N_{K,\alpha} - 2, N_{K,\alpha} - 1, N_{K,\alpha} + 1),
    (N_{K,\alpha} -1, N_{K,\alpha} - 1, N_{K,\alpha} + 1),\\
    &(N_{K,\alpha}, N_{K,\alpha} - 1, N_{K,\alpha} + 1),
    (N_{K,\alpha}, N_{K,\alpha}, N_{K,\alpha} + 1) \big),    
\end{align*}
and $\gamma^2_x$ is
\begin{align*}
    \big(&(N_{K,\alpha}, N_{K,\alpha}, N_{K,\alpha} + 1),(N_{K,\alpha}-1, N_{K,\alpha}, N_{K,\alpha} + 1),(N_{K,\alpha}-2, N_{K,\alpha}, N_{K,\alpha} + 1),\\
    &(N_{K,\alpha}-3, N_{K,\alpha}, N_{K,\alpha} + 1), (N_{K,\alpha}-3, N_{K,\alpha}-1, N_{K,\alpha} + 1), (N_{K,\alpha}-3, N_{K,\alpha}-2, N_{K,\alpha} + 1),\\
    &(N_{K,\alpha}-3, N_{K,\alpha}-3, N_{K,\alpha} + 1),(N_{K,\alpha}-3, N_{K,\alpha}-3, N_{K,\alpha}),
    (N_{K,\alpha}-3, N_{K,\alpha}-3, N_{K,\alpha}-1),\\
    &(N_{K,\alpha}-3, N_{K,\alpha}-3, N_{K,\alpha}-2) \big).
\end{align*}

   Returning to the general case, we now observe the following.
\begin{itemize}
    \item For each $x\in D$, we have $|\gamma_x| = \lceil \frac{3}{\alpha} \rceil d$.  For $x \notin D$, we have  $|\gamma^1_x|\le N_{K,\alpha} d$ and $|\gamma^2_x|=  \lceil \frac{3}{\alpha} \rceil d$.  Thus,  $|\gamma_x|= |\gamma^1_x|+ |\gamma^2_x| \le N_{K,\alpha} d+\lceil \frac{3}{\alpha} \rceil d$.   Therefore, condition \eqref{eq:overlineL} holds with $\overline L = N_{K,\alpha} d+\lceil \frac{3}{\alpha} \rceil d.$
    \item Each edge in $\gamma(x,t(x))$ must be within either $\lceil \frac{3}{\alpha} \rceil d+1$ transitions (if $x\in D$) or $N_{K,\alpha} d+ \lceil \frac{3}{\alpha} \rceil d$ transitions (if $x\in D^c$) of $x$. Hence the number of such paths that a particular edge can be contained within is bounded and \eqref{eq:no infinitely many paths intersects} holds. 
\end{itemize}

    Now we show that condition \eqref{eq:supmaxassump} holds. For $x\in D$, it holds that 
    \begin{align}\label{eq:condition 20 for inside D}
        \sup_{x\in D}\frac{\pi(x)}{\pi_{\text{min}}(\gamma_x)}
        %\min_{z\in \gamma_x}\pi(z)} 
        \le 1
    \end{align}
     since for $x\in D$, we used the same construction of $\gamma_x$ as  in the proof of Corollary \ref{cor:main}. Now we suppose that $x\in D^c$. By the construction of $\gamma^1_x$ and by the assumption that the distribution $\pi$ is of product-form, for any state $x_m$, with $m\ge 2$, contained within the path $\gamma^1_x$ we have
   \begin{align*}%\label{eq:ratio in Dc}
       \frac{\pi(x_{m-1})}{\pi(x_m)} = \frac{\pi_{i}(x_{m-1,i})}{\pi_{i}(x_{m-1,i}+1)} \quad \text{for some $i$ for which  $x_i<N_{K,\alpha}$,}
   \end{align*}
  which is uniformly bounded, and where where we denoted the $i$ th component of $x_{m-1}$ by $x_{m-1,i}$.
   Hence,  
   \begin{align}
   \label{eq:8976976}
    \sup_{x\in D^c} \dfrac{\pi(x)}{\pi_{\text{min}}(\gamma^1_x)}
    %\min_{z\in \gamma^1_x}\pi(z)}
    < \infty.
    \end{align}
    
   Next,
 \begin{align}
 \begin{split}
 \label{eq:23456711}
     \sup_{x\in D^c}\frac{\pi(x)}{\pi_{\text{min}}(\gamma^2_x)}  %\min_{z\in \gamma^2_x} \pi(z)} 
     &= \sup_{x\in D^c} \frac{\pi(x)}{\pi(t_{\text{inter}}(x))}   \cdot  \frac{\pi(t_{\text{inter}}(x))}{ \pi_{\text{min}}(\gamma^2_x) }\\
     %\min_{z\in \gamma^2_x}\pi(z)}\\
     &\le \sup_{x\in D^c} \frac{\pi(x)}{\pi(t_{\text{inter}}(x))} \cdot   \sup_{x\in D^c} \frac{\pi(t_{\text{inter}}(x))}{ \pi_{\text{min}}(\gamma^2_x)}
     %\min_{z\in \gamma^2_x}\pi(z)} 
     < \infty 
     \end{split}
 \end{align}  
 where the term $\sup_{x\in D^c} \frac{\pi(x)}{\pi(t_{\text{inter}}(x))}$ is finite due to \eqref{eq:8976976}, and the term $\sup_{x\in D^c} \dfrac{\pi(t_{\text{inter}}(x))}{\pi_{\text{min}}(\gamma^2_x)}$ 
 %\min_{z\in \gamma^2_x}\pi(z)}$ 
 is also finite by the same argument used in the proof of Corollary \ref{cor:main}. Combining \eqref{eq:8976976} and \eqref{eq:23456711} yields
 \begin{align}\label{eq:condition 20 for outside D}
        \sup_{x\in D^c}\frac{\pi(x)}{\pi_{\text{min}}(\gamma_x)}
        %\min_{z\in \gamma(x)}\pi(z)}
        < \infty.
    \end{align}
Combining \eqref{eq:condition 20 for inside D} and \eqref{eq:condition 20 for outside D} implies that \eqref{eq:supmaxassump} holds.

We now construct a set of paths $\gamma(t(x),t(x'))$ that satisfies \eqref{eq:sum_condition}. Let $N_{K} = N+K+1$ and let
\[
    D_{K} := \{x\in \Z^d_{\ge 0} : x_i \ge N_{K} \text{ for each $i$}\}.
\]
Note that for any $x$ we have   $t(x)\in D_{K}$, and that for any $y\in D_{K}$ all of the birth and death rates for each index $i \in \{1,\dots,d\}$ at $y$ are lower bounded by $c_{\text{min}}$ due to conditions 2 and 3 in the statement of the corollary.  Hence for each pair of  states $x$ and $x'$, we use exactly the same construction of $\gamma(t(x),t(x'))$ that we used in the proof of Corollary \ref{cor:main}.
We highlight that for any $z\in \gamma(t(x),t(x')) $ we have
\begin{align*}
    \min\{t(x)_i, t(x')_i\} \le z_i \le \max\{t(x)_i, t(x')_i\},
\end{align*}
 by the construction of the path $\gamma(t(x),t(x'))$. Hence $z\in D_{K}$ and  the birth and death rates for each index $i$ is lower bounded by $c_{\text{min}}$.  Therefore $\gamma(t(x),t(x'))$ is always an active path.

Now we show that \eqref{eq:sum_condition} holds with the set of paths $\Gamma_{\mathcal T}$ we constructed above.  First note that for each $y\in D$ we used  exactly the same construction of $\gamma_y$ as in the proof of Corollary \ref{cor:main}. Hence, by the same argument as in the proof of Corollary \ref{cor:main}, and the fact that the path $\gamma(t(y),t(y'))$ below is necessarily active, we have that
\begin{align}\label{eq:by the previous cor}
    \sum_{\substack{(y,y')\in D\times D \\ y\neq y'}} |\gamma(t(y),t(y'))| \left( \frac{\pi(y)\pi(y')}{\pi_{\text{min}}(\gamma(t(y),t(y')))}\right)
    %\min_{z\in \gamma(t(y),t(y'))}\pi(z)
    < \infty,
\end{align}
where it is important to notice that the sum is over pairs $D\times D$ and not $\Z^d_{\ge 0} \times \Z^d_{\ge 0}$.

To expand the sum to all of $\Z^d_{\ge 0} \times \Z^d_{\ge 0}$ we simply consider which states $x\in \Z^d_{\ge 0}$ get mapped to $y\in D$ via $t_{\text{inter}}$.  Hence, for $y \in D$, we define $B_y=\{x\in \mathbb Z^d_{\ge 0} : y=t_{\text{inter}}(x) \}$.  Note that there is a positive integer $c\ge 1$ for which 
\begin{align}\label{eq:finite By}
    \sup_{y\in D} |B_y| :=  c < \infty.
\end{align}
 %Note that $B_y$ is always non-empty since $y\in B_y$.  
 Next, note that by \eqref{eq:8976976} there exists a $c'>0$ so that for any $y\in D$ and any $x \in B_y$ we have $\pi(x)\le c' \pi(t_{\text{inter}}(x)) = c'\pi(y)$.  Thus,
 \begin{align*}
&\sum_{\substack{(x,x')\in \Z^d_{\ge 0}\times \Z^d_{\ge 0}\\
    x\ne x'}} |\gamma(t(x),t(x'))| \left( \frac{\pi(x)\pi(x')}{\pi_{\text{min}}(\gamma(t(x),t(x')))}\right)\\
    %\min_{z\in \gamma(t(x),t(x'))} \pi(z)}\right)\\
\le 
&c(c')^2\sum_{\substack{(y,y')\in D \times D \\ y\neq y'}} |\gamma(t(y),t(y'))| \left( \frac{\pi(y)\pi(y')}{\pi_{\text{min}}(\gamma(t(y),t(y')))}\right),
%\min_{z\in \gamma(t(y),t(y'))} \pi(z)}\right),
\end{align*}
 which is finite by \eqref{eq:by the previous cor} (and we remind that $t(x) = t(y)$ for any $x \in B_y$.  Hence, \eqref{eq:sum_condition} holds.

Finally, note that \eqref{eq:activepathcond} holds due to the existence of $c_{\text{min}}$ in the statement of Corollary \ref{cor:relaxed}.  Hence, the proof is complete.
\end{proof}

\begin{example}
    Let $X(t)=(X_1(t),X_2(t),X_3(t))$ be a Markov chain modeling the number of customers in queues $A$, $B$, and $C$, respectively. We assume that the customers who have been served at queue $A$ immediately move to $B$, and then move to $C$ when the service at $B$ is completed. Once a customer at queue $C$ is served, the customer leaves the queuing system. The service time at each queue $A, B$ and $C$ is exponentially distributed with the rates $\lambda_A$, $\lambda_B$, and $\lambda_C$, respectively. The arrival time at queue $A$ is also exponentially distributed with the rate $\lambda_0$. We represent the system pictorially via the image below.
    \begin{align*}
        \xrightarrow{\quad \lambda_0 \quad }\boxed{A} \xrightarrow{\quad  \lambda_A \quad } \boxed{B} \xrightarrow{\quad  \lambda_B \quad }  \boxed{C} \xrightarrow{\quad  \lambda_C  \quad }
    \end{align*}

We further assume that the number of servers on each queue is infinite (i.e.~the system is an M/M/$\infty$ system). Hence, the rate of service at each queue is proportional to the current number of customers; this implies that the transition rate  from queue $A$ to $B$, $B$ to $C$, and $C$ to outside the system at time $t$ is $\lambda_A X_1(t), \lambda_B X_2(t)$, and $\lambda_C X_3(t)$, respectively.  By well known results (see one of \cite{Whittle86, AndProdForm,Kelly1979}) the stationary distribution for this model is a product  of Poissons  
\begin{align*}    \pi(x)= e^{-(\frac{\lambda_0}{\lambda_A} + \frac{\lambda_A}{\lambda_B}+\frac{\lambda_{B}}{\lambda_C})}\frac{(\lambda_0/\lambda_A)^{x_1}}{x_1!} \cdot \frac{(\lambda_A/\lambda_B)^{x_2}}{x_2!} \cdot \frac{(\lambda_B/\lambda_C)^{x_3}}{x_3!}, \quad x \in \Z^d_{\ge 0}. 
\end{align*}  
Hence,  condition 1 in Corollary \ref{cor:relaxed} holds.
Furthermore, if we let $J_0=\{1\}, J_1=\{2\}$ and $J_2=\{3\}$, then conditions 2 and 3 in Corollary \ref{cor:relaxed} also hold. 
Consequently,  by Corollary \ref{cor:relaxed}, the conditions of Assumption \ref{assump:main} hold. So, by Theorem \ref{thm:main for poincare} the process $X$ is exponentially ergodic and the mixing time $\tau^x_\varepsilon$ is bounded above by $C|\ln(\pi(x))|$ for some constant $C>0$.
\hfill $\triangle$
 \end{example}

\subsection{Key Examples}\label{subsec:canonical}

In this section, we analyze two examples from our motivating class of models: stochastic reaction networks (which will be formally introduced in Section \ref{sec:modelsandresults}).  The first example, Example \ref{ex:key example}, demonstrates the necessity of our new theory to handle our models of interest. In particular, we show that an existing geometric method (related to the ``canonical path method'') is not powerful enough to prove exponential ergodicity for the provided model, but our theory--specifically Corollary \ref{cor:relaxed}--is able to prove that the spectral gap is strictly positive.
In our second example, Example \ref{ex:non expo}, we provide a time-reversible stochastic reaction network (which is also detailed-balanced in the sense of \cite{HornJack72}) and prove that it is \textit{not} exponentially ergodic.  This example demonstrates the necessity of our assumption 2 in Corollary \ref{cor:main} and assumptions 2 and 3 in Corollary \ref{cor:relaxed} (that transitions of the form $\pm e_i$ have positive rates under certain conditions).  Moreover, we believe this to be the first detailed-balanced model (and, hence, the first complex-balanced model, see Definition \ref{def:CB} and Theorem \ref{thm:def0}) in the stochastic reaction network literature to be proven to \textit{not} be exponentially ergodic.

Before moving to Example \ref{ex:key example}, we  briefly discuss some existing theory pertaining to the \textit{canonical path method} (see for instance \cite[Chapter 3]{saloff1997lectures} or \cite[Section 3.3]{berestycki2009eight}), which was the motivating theory for the present work.
Consider a positive recurrent continuous-time Markov chain $X$ with irreducible state space $\mathbb{S} \subset \Z^d$ and stationary distribution $\pi$.  We again denote by $q(z,w)$ the transition rate from state $z$ to state $w$. We recall that we  denote  a directed edge between $z$ and $w$ by $(z,w)$, and we write $(z,w) \in \gamma(x,x')$ if $(z,w)$ is an edge in the path $\gamma(x,x')$.  
Similarly to as we did in Assumption \ref{assump:main}, we let
\begin{align}\label{eq:654433578}
\Gamma =  \{\gamma(x,x') \ \text{ or }\  \gamma(x',x) \text{ but not both} : x, x'\in \mathbb{S} \text{ with } x \ne x' \}
\end{align}
that contains precisely one active, directed path for each pair of  states.   
We then define
the so-called ``congestion ratio", 
\begin{align}\label{eq:canonical}
 C_{\rm cr}(\Gamma):=\sup_{(z,w)} \left ( \frac{1}{q(z,w)\pi(z)} \sum_{\substack{x,x' :\\ (z,w) \in \gamma(x,x')}} |\gamma(x,x')|\pi(x)\pi(x') \right ),
\end{align}
where the supremum is taken over all directed edges $(z,w)$ and the sum is over pairs of states, $x,x'\in \mathbb{S}$, for which $(z,w)$ is an edge in $\gamma(x,x')$. Note that if $(z,w)\in \gamma(x,x')$, then $q(z,w)>0$ because the path is active.    Note that, very loosely, we should expect $C_{\rm cr}(\Gamma)$ to be  large if there is an edge, $(z,w)$, which many paths go through.

We can now state the basic idea of the method. If there exists a set of paths $\Gamma$ for which $C_{\text{cr}}(\Gamma) < \infty$, then the spectral gap is lower bounded by $1/C_{\text{cr}}(\Gamma)$.  
Of course, if $C_{\text{cr}}(\Gamma)=\infty$ then there is no hope of using this method and it stands silent.

\begin{example}\label{ex:key example}
Consider a continuous-time Markov chain defined on $\mathbb Z^2_{\ge 0}$ with  transition rates given as follows: denoting the elementary basis vectors by $e_1=(1,0)^\top$ and $e_2=(0,1)^\top$, for each $x=(x_1,x_2)^\top \in\mathbb Z^2_{\ge 0}$
\begin{align*}
&q(x,x+e_1)=\kappa_1 x_2,\\
&q(x,x-e_1)=\kappa_2 x_1x_2,\\
&q(x,x+e_2)=\kappa_3,\\
&q(x,x-e_2)=\kappa_4 x_2,
\end{align*}
for some positive constants $\kappa_1, \kappa_2, \kappa_3, \kappa_4\in\R_{>0}$. In Example \ref{ex: key example of crn} of Section \ref{sec:modelsandresults} we will see that this model arises via consideration of the stochastic reaction network with associated graph
   \begin{equation*}
   \begin{tikzpicture}[baseline={(current bounding box.center)}]
   \node[state] (X1+X2) at (-2.7,0) {$X_1+X_2$};
   \node[state] (X2)  at (0,0) {$X_2$};
   \node[state] (0)  at (2,0) {$0$};
    \path[->]
     ([yshift=-0.3em]X2.west) edge[bend left] node[below] {$\kappa_1$} ([yshift=-0.3em]X1+X2.east)
     ([yshift=0.3em]X1+X2.east) edge[bend left] node {$\kappa_2$} ([yshift=0.3em]X2.west)
     ([yshift=-0.3em]0.west) edge[bend left] node[below] {$\kappa_3$} ([yshift=-0.3em]X2.east)
     ([yshift=0.3em]X2.east) edge[bend left] node {$\kappa_4$} ([yshift=0.3em]0.west);
   \end{tikzpicture}
\end{equation*}
% \begin{align*}
%     X_1+X_2 \xrightleftharpoons[1]{1} & \ \ X_2 \  \xrightleftharpoons[1]{1} \emptyset.
% \end{align*}
Biologically, this system would model the creation and destruction of a molecule/protein of type $X_2$, that acts as a catalyst for the creation and destruction of a molecule/protein of type $X_1$. $\Z^2_{\ge 0}$ is irreducible and by Theorem \ref{thm:def0} (originally found in \cite{AndProdForm}), the model has stationary distribution
\begin{align}\label{eq:stationary}
    \pi(x) = e^{-\kappa_1/\kappa_2-\kappa_3/\kappa_4} \cdot \frac{(\kappa_1/\kappa_2)^{x_1}}{x_1!}\cdot \frac{(\kappa_3/\kappa_4)^{x_2}}{x_2!}, \quad x\in \Z^2_{\ge 0}.
\end{align}
It is straightforward to check that this model is  time-reversible (and  detailed balanced in the sense of \cite{HornJack72}, which implies time-reversibility but is not equivalent to it \cite{joshi2015detailed, cappelletti2018graphically}). 

The Lyapunov function techniques developed in \cite{mixing_And_Kim} to demonstrate exponential ergodicity for stochastic reaction networks do not work in a straightforward manner for this model (these authors could not develop a suitable Foster-Lyapunov function).  Thus, we turn to spectral  methods.  We will show that the canonical path method, as discussed in this sub-section above, stands silent for this model, but that our new techniques are strong enough to prove the desired result.

To show that the canonical path method stands silent, it suffices to show that $C_{\rm cr}(\Gamma)=\infty$ for {\it any} choice of paths $\Gamma$.
Hence, we let $\Gamma$ be an arbitrary set of active paths of the form \eqref{eq:654433578}.  For each positive integer $n\ge1$ we let $z_n = (n,0)^\top$ and $w_n = (n,1)^\top$.
Finally, we let $y_0\in \Z^2_{\ge 0}$ be a fixed state.

Note that $\Gamma$ must contain an infinite number of active paths from a set $\{\gamma(z_n,y_0)\}_{n\ge 1}$ and/or an infinite number of active paths from a set $\{\gamma(y_0,z_n)\}_{n\ge 1}$. Also note that each such active path $\gamma(z_n,y_0)$ contains the directed edge $(z_n,w_n)$ whereas each active path $\gamma(y_0,z_n)$ contains the directed edge $(w_n,z_n)$.  This follows because the only transition  starting at  $z_n=(n,0)^\top$ ends in $w_n=(n,1)^\top$, and the only transition ending in $z_n=(n,0)^\top$ starts at $w_n = (n,1)^\top$.

 We first assume that $\Gamma$ contains an  infinite number of active paths from a set $\{\gamma(z_n,y_0)\}_{n \ge 1}$. Denoting these active paths by $\gamma(z_{n_k},y_0)$ with a sub-sequence $n_k$ such that $n_k\to \infty$, as $k\to \infty$, we have (by only considering the one particular edge $(z_{n_k},w_{n_k})$)
\begin{align*}
    C_{\rm cr} (\Gamma) &\ge \frac{1}{q(z_{n_k},w_{n_k})\pi(z_{n_k})}|\gamma(z_{n_k},y_0)|\pi(z_{n_k})\pi(y_0) \\
    &=\frac{1}{\kappa_3}|\gamma(z_{n_k},y_0)|\pi(y_0),
\end{align*}
where we use that $q(z_{n_k},w_{n_k}) = \kappa_3$ for each $n_k$.
The right-hand side necessarily goes to $\infty$ as $k \to \infty$, since $|\gamma(z_{n_k},y_0)|\to \infty$ and $\pi(y_0)$ is fixed. In the same way, we can also show $C_{\rm cr}(\Gamma)=\infty$ for the  case when $\Gamma$ contains  infinitely many active paths from the set $\{\gamma(y_0,z_n)\}_{n \ge 1}$.  Hence %\textit{for any choice of paths} $\Gamma$, 
we have that the canonical path method stands silent for this model.

 However, the main method we propose in this paper successfully shows exponential ergodicity of this model. First, condition 1 in Corollary \ref{cor:relaxed} clearly holds because of the form of the stationary distribution given in \eqref{eq:stationary}.   Also conditions 2 and 3 in Corollary \ref{cor:relaxed} hold with  $J_0=\{ 2\}, J_1=\{ 1\}$, and $N=1$. \hfill $\triangle$
\end{example}

 We turn to our next example, demonstrating the necessity of our key assumptions.

\begin{example}\label{ex:non expo}
    Consider a continuous-time Markov chain with the transition rates given as follows: for each $x=(x_1,x_2)^\top \in  \mathbb Z^2_{\ge 0}$ 
    \begin{align*}
        &q(x,x+e_1+e_2)=1,\\
        &q(x,x-e_1-e_2)=x_1x_2,\\
        &q(x,x+e_2)=x_2, \text{ and}\\
        &q(x,x-e_2)=x_2(x_2-1).
    \end{align*}
    Note that $\Z^d_{\ge 0}$ is irreducible.  This model arises as the  stochastic mass-action system of Section \ref{sec:modelsandresults} associated to the reaction graph 
   \begin{equation*}
   \begin{tikzpicture}[baseline={(current bounding box.center)}]
   \node[state] (X1+X2) at (2.3,0) {$X_1+X_2$};
   \node[state] (0)  at (0,0) {$0$};
   \node[state] (X2)  at (4.7,0) {$X_2$};
   \node[state] (2X2)  at (6.7,0) {$2X_2$};
    \path[->]
     ([yshift=-0.3em]X1+X2.west) edge[bend left] node[below] {$1$} ([yshift=-0.3em]0.east)
     ([yshift=0.3em]0.east) edge[bend left] node {$1$} ([yshift=0.3em]X1+X2.west)
     ([yshift=-0.3em]2X2.west) edge[bend left] node[below] {$1$} ([yshift=-0.3em]X2.east)
     ([yshift=0.3em]X2.east) edge[bend left] node {$1$} ([yshift=0.3em]2X2.west);
   \end{tikzpicture}
\end{equation*}
   % \begin{align*}
   %      \emptyset \xrightleftharpoons[1]{1} X_1+X_2, \quad X_2 \xrightleftharpoons[1]{1} 2X_2.
   %  \end{align*}
     Theorem \ref{thm:def0} implies that the stationary distribution of this model is
\begin{align}\label{eq:stationary2nd}
    \pi(x) = e^{-2} \frac{1}{x_1!}\cdot \frac1{x_2!}, \quad x\in \Z^2_{\ge 0}.
\end{align}

It is straightforward to check that this  Markov chain is  time-reversible. To do so, one just needs to verify the following two equalities for all $x \in \Z^2_{\ge 0}$
\begin{align*}
    \pi(x)q(x,x+e_1+e_2) &= \pi(x+e_1+e_2)q(x+e_1+e_2,x),\\
    \pi(x)q(x,x+e_2) &= \pi(x+e_2)q(x+e_2,x)
\end{align*}
Verification is straightforward, and is left to the reader. It is further not complicated (by using techniques of reaction network theory such as deficiency and the tree-like structure of the network \cite{feliu201868,dickenstein2011far,muller2020detailed}) to show that the model is detailed balanced in the sense of \cite{HornJack72}, which implies time-reversibility but is not equivalent to it \cite{joshi2015detailed, cappelletti2018graphically}.  By Theorem \ref{thm:adlkjfkd}, it is enough to demonstrate that $\text{gap}(\EE) =0$ to prove the model is not $L^2$ exponentially ergodic. (Moreover, because the model is time-reversible, even the weaker form of exponential convergence given by (iii) in Theorem \ref{thm:adlkjfkd} does not hold.) Hence, by \eqref{eq:2222NEW}, it is sufficient to show that 
\begin{align}\label{eq:2222}
     \frac12 \inf\left\{ \sum_{x,y\in \Z^2_{\ge 0}} q(x,y) [ f(y) - f(x)]^2 \pi(x) : f \in \mathcal{B}_c,\ \pi(f) = 0,\ \pi(f^2) = 1\right\}=0,
\end{align}
where, as always, $\mathcal{B}_c$ is the set of functions with bounded support.

Let 
\[
    f_n(x) = c_nI_n(x) - d_n,
\]
where $I_n$ is the indicator function on $\{(n,0),(n+1,1)\}$ and $c_n$ and $d_n$ will be chosen so that $\pi(f_n) = 0$ and $\pi(f_n^2)=1$.  
It is straightforward to check that these conditions imply
\begin{align*}
    c_n^2 = \frac{1}{\pi(n,0)+\pi(n+1,1) - (\pi(n,0)+\pi(n+1,1))^2}.
\end{align*}
From \eqref{eq:stationary2nd}, the dominant term in the denominator is $\pi(n,0) = e^{-2}\frac1{n!}$.  Hence, we conclude that
\[
c_n^2 \approx e^2 \cdot n!. 
\]
Now we compute the term in \eqref{eq:2222} with $f_n$. We just have to consider those transitions taking us in to and out of the set $\{(n,0),(n+1,1)\}$. We have (the first two are the only possible transitions out of the set and the next two are the two transitions that can take you in):
\begin{align}
\begin{split}\label{eq:9u687656785}
    q&((n+1,1),(n+2,2)) [f_n(n+2,2)-f_n(n+1,1)]^2\pi(n+1,1)\\
    &+ q((n+1,1),(n+1,2)) [ f_n(n+1,2)-f_n(n+1,1)]^2 \pi(n+1,1)\\
    &+ q((n+2,2),(n+1,1)) [ f_n(n+1,1)-f_n(n+2,2)]^2\pi(n+2,2)\\
    &+q((n+1,2),(n+1,1))[ f_n(n+1,1)-f_n(n+1,2)]^2 \pi(n+1,2)\\
    % =&\ 1\cdot [-d_n -(c_n-d_n)]^2\pi(n+1,1)\\
    % &+ 1[ -d_n-(c_n-d_n)]^2 \pi(n+1,1)\\
    % &+ 2(n+2)   [(c_n-d_n)-(-d_n)]^2\pi(n+2,2)\\
    % &+2 [ (c_n-d_n)- (-d_n)]^2 \pi(n+1,2)\\
     =&\ c_n^2\left[ 2 \pi(n+1,1)+ 2(n+2)   \pi(n+2,2)+2 \pi(n+1,2)\right].
     \end{split}
\end{align}
Thus, combining $c_n^2 \approx e^2 n!$ with \eqref{eq:stationary2nd}, we have that the term above satisfies
\begin{align*}
    &\approx 2 \cdot n! \frac{1}{(n+1)!} + 2(n+2) \cdot n! \frac{1}{2(n+2)!} + 2 \cdot n! \frac{1}{2(n+1)!},
\end{align*}
which goes to zero, as $n\to \infty$.  Hence,  the infimum of \eqref{eq:2222} is zero.  Thus, for this particular model, $\text{gap}(\EE) = 0$ and it is not exponentially ergodic.

Now we can ask: what would have been different in the above calculation if we further allow transitions of the form $\pm e_i$ (i.e., what if the model satisfied the assumptions of our Corollaries \ref{cor:main} and \ref{cor:relaxed})? Because of the presence of a rate for $+e_1$ that is lower bounded (by, say, $\lambda >0$), we would get the following term added to \eqref{eq:9u687656785} 
\begin{align*}
    &q((n,0),(n+1,0)) [ f_n(n+1,0)-f_n(n,0)]^2 \pi(n,0)\ge \lambda c_n^2 \pi(n,0)\approx \lambda \cdot n! \cdot \frac1{n!} = 1.
\end{align*}
Hence, we could not have concluded that the gap was zero in this new scenario. 
Of course, in this new scenario we \textit{know} the gap is strictly positive by our Corollaries \ref{cor:main} and \ref{cor:relaxed}.
    \hfill $\triangle$
\end{example}

\section{Applications to stochastic reaction networks}
\label{sec:modelsandresults}

 In this section we explore the applicability of our results to a wide class of continuous-time Markov chains used heavily in biochemistry, ecology, and epidemiology. The models are referred to as \emph{stochastic reaction networks}. We will provide the necessary definition of the models under consideration in Section~\ref{sec:basic}, then examine a specific example, and finally prove exponential ergodicity for large families of stochastic reaction networks that are of particular interest to the community.

\subsection{Reaction networks and their associated stochastic model}
\label{sec:basic}

We give a minimal introduction to reaction networks and their associated dynamical systems.  For a more thorough introduction, see \cite{AK2015}.
We begin with the definition of a reaction network.  

\begin{defn}\label{def:21}
\emph{A  reaction network} is given by a triple of finite sets $(\Sp,\C,\Re)$ where:
\begin{enumerate}
\item \emph{The species set} $\Sp=\{X_1,X_2,\cdots,X_d\}$ contains the species of the reaction network.
\item \emph{The reaction set} $\Re=\{R_1,R_2,\cdots,R_r\}$ consists of ordered pairs $(y,y') \in \Re$, with $y\ne y'$, where 
\begin{align}\label{complex}
y=\sum_{i=1}^d y_iX_i \hspace{0.4cm} \textrm{and} \hspace{0.4cm}
y'=\sum_{i=1}^d y'_iX_i
\end{align}
and where the values $y_i,y'_i \in \mathbb{Z}_{\ge 0}$ are the \emph{stoichiometric coefficients}. We will  write reactions $(y,y')$ as $y\rightarrow y'$, and refer to $y$ as the \textit{source complex} and $y'$ as the \textit{product complex}.
\item \emph{The complex set} $\C$ consists of the linear combinations of the species in (\ref{complex}). Specifically, 
$\C = \{y\ : \ y\rightarrow y' \in \Re\} \cup \{y' \ : \ y\rightarrow y' \in \Re\}$.
\hfill $\triangle$
\end{enumerate}
\end{defn}

Note that when we are working in a general setting we are denoting our species via $X_1,\dots, X_d$.  However, when a specific model is under consideration, such as that found in Figure \ref{figure:EnvZ_model1}, more suggestive notation is often utilized (for example, in that model $X_p$ and $Y_p$ are the ``p''hosporolated versions of certain proteins).

Depending on the context, $y$ can denote either the linear combination in \eqref{complex} or the vector whose $i$-th component is $y_i$,  i.e.~$y=(y_1,y_2,\cdots,y_d)^\top \in \mathbb{Z}^d_{\ge 0}$. 
For example, when $\Sp=\{X_1,X_2,\dots,X_d\}$, $y=2X_1+X_2$ is associated with $(2,1,0,0,\dots,0)^\top \in \Z^d_{\ge 0}$. 
When $y=(0,0,\dots,0)^\top$, we denote the complex by $\emptyset$. This complex is used for the death of species, such as $X_1 \to \emptyset$, or the production of species from outside of the system, such as $\emptyset \to X_1$.

The triple of sets found in Definition \ref{def:21} is most often  illustrated with a \textit{reaction graph} in which the nodes are the complexes, each complex is written exactly one time (even if it appears in more than one reaction), and the directed edges are given by the reactions.

\begin{example}
The  reaction graph
\[
	S+E \rightleftarrows SE \rightarrow E+P,
\]
has $\Sp = \{S,E,SE,P\}$, $\C=\{S+E,SE,E+P\}$ and $\Re=\{S+E\rightarrow SE, SE\rightarrow S+E,SE\rightarrow E+P\}$. This is a standard model for substrate-enzyme kinetics. \hfill $\triangle$
\end{example}

\subsubsection{Stochastic model}

Given a reaction network $(\Sp,\C,\Re)$, a \emph{(stochastic) kinetics} is an assignment of a rate or \textit{intensity} function $\lambda_{y\to y'}:\Z_{\geq0}^d\to \R_{\geq0}$ to each reaction $y\to y'\in \Re$. We then call $(\Sp,\C,\Re, \{\lambda_{y\to y'}\})$ a \emph{stochastic reaction system}.    The change in  species counts is modeled by means of a continuous-time Markov chain with state space $\Z_{\geq0}^d$, whose transition rates are given by
\begin{align*}
  q(x,x')=\sum_{\substack{y\to y'\in\Re\\ y'-y=x'-x}}\lambda_{y\to y'}(x),
\end{align*}
where the sum is over those reactions whose occurrence causes a net change that is precisely $x'-x$. 
In case of an explosion occurring at a finite-time $T_\infty$, we consider  $X(t)=\Delta$ for any $t\geq T_\infty$, where $\Delta$ is a cemetery state not contained in $\Z_{\geq 0}^d$ \cite{NorrisMC97}.  The infinitesimal generator $\mathcal{A}$ of the associated Markov process acts on functions via the operation
\begin{align}\label{gen5}
 \mathcal{A}f(x) &=\sum_{x'}q(x,x')(f(x')-f(x)) \notag\\
 =& \sum_{y \rightarrow y' \in \Re} \lambda_{y\rightarrow y'}(x)(f(x+y'-y)-f(x)),
\end{align}
for any function $f$ in its domain  satisfying $f(\Delta)=0$.

A popular choice of stochastic kinetics is given by  \emph{(stochastic) mass-action kinetics}, where for any reaction $y\to y'\in\Re$
\begin{align}\label{mass}
\lambda_{y\to y'}(x)=\kappa_{y\to y'} \prod_{i=1}^d \frac{x_i!}{(x_i-y_i)!} \mathbbm{1}_{\{x_i \ge y_i\}},
\end{align}
for reaction constant $\kappa_{y\to y'}$. We denote $\K=\{\kappa_{y\to y'}\}$.   When the intensity functions are given by mass-action kinetics  we write the stochastic system as $(\Sp,\C,\Re,\K)$.

Note that the first three conditions of Assumption \ref{intcnd} automatically hold for mass-action systems. However, the fourth condition (non-explosivity) may not.

\subsection{Exponential ergodicity for stochastic reaction systems}

We start this section with an example on how to use our results in the context of stochastic reaction networks. We then proceed to state general results pertaining models related to complex-balancing (but not necessarily complex-balanced).

\begin{example}\label{ex:auto}
   Consider the following mass-action system, with the rate constants written next to the associated reaction:
   \begin{equation*}
   \begin{tikzpicture}[baseline={(current bounding box.center)}]
   \node[state] (X1) at (-2,0) {$X_1$};
   \node[state] (0)  at (0,0) {$0$};
   \node[state] (X2)  at (2,0) {$X_2$};
    \node[state] (X1+X2)  at (6,0) {$X_1+X_2$};
    \node[state] (2X1)  at (4,0) {$2X_1$};
    \node[state] (2X2)  at (8,0) {$2X_2$};
    \path[->]
     (0) edge[bend left] node[below] {$\kappa_1$} (X1)
     (X1) edge[bend left] node {$\delta$} (0)
     (0) edge[bend left] node {$\kappa_2$} (X2)
     (X2) edge[bend left] node[below] {$\delta$} (0)
     (X1+X2) edge node {$\rho$} (2X1)
     (X1+X2) edge node {$\rho$} (2X2);
   \end{tikzpicture}
\end{equation*}
This is model where particles of both types promote their own creations via the auto-catalytic reactions $X_1+X_2\to 2X_1$ and $X_1+X_2\to 2X_2$. Auto-catalytic models have been extensively studied in biochemistry and related mathematical questions have been popularised in \cite{TK2001}. Here we assume that the degradation rates of both types of particles is the same, which is not unrealistic in many concrete examples where degradation is due to homogeneous dilution. We further assume that the two molecules are equally likely to be produced via the auto-catalytic mechanism. The above example has been studied in \cite{bibbona2020stationary}, where it is proved that the model is positive recurrent and the stationary distribution is given by
\begin{equation*}
    \pi(x)=\frac{M}{x_1!x_2!}\frac{\Gamma(x_1+\gamma_1)\Gamma(x_2+\gamma_2)}{\Gamma(x_1+x_2+\gamma_1+\gamma_2)}\left(\frac{\kappa_1+\kappa_2}{\delta}\right)^{x_1+x_2},
\end{equation*}
where $\Gamma$ denotes the Gamma function,
\begin{equation*}
\gamma_1=\frac{\delta\kappa_1}{\rho(\kappa_1+\kappa_2)},\quad\gamma_2=\frac{\delta\kappa_2}{\rho(\kappa_1+\kappa_2)},
\end{equation*}
and $M$ is a normalizing constant that can be explicitly calculated as
\begin{equation*}
    M=\frac{\Gamma(\gamma_1+\gamma_2)}{\Gamma(\gamma_1)\Gamma(\gamma_2)}e^{-\frac{\kappa_1+\kappa_2}{\delta}}.
\end{equation*}
Note that in this case, for each $x\in\Z^2_{\geq0}$ with $x_1\geq1$ we have
\begin{align*}
 \frac{\pi(x)}{\pi(x-e_1)}&=\frac{1}{x_1}\frac{x_1+\gamma_1-1}{x_1+x_2+\gamma_1+\gamma_2-1}\frac{\kappa_1+\kappa_2}{\delta} \leq\frac{1}{x_1}\frac{\kappa_1+\kappa_2}{\delta}.
\end{align*}
In particular, if $\sqrt{x_1}\geq (\kappa_1+\kappa_2)/\delta$ then
\begin{equation*}
    \frac{\pi(x)}{\pi(x-e_1)}\leq \frac{1}{\sqrt{x_1}}.
\end{equation*}
Similarly, if $x_2\geq \min\{1, (\kappa_1+\kappa_2)^2/\delta^2\}$ we have
\begin{equation*}
    \frac{\pi(x)}{\pi(x-e_2)}\leq \frac{1}{\sqrt{x_2}}.
\end{equation*}
Hence, by Corollary~\ref{cor:main} we can conclude that the model is exponentially ergodic, the second condition of Corollary~\ref{cor:main} being guaranteed by the presence of the reactions $X_1\rightleftharpoons 0 \rightleftharpoons X_2$.\hfill $\triangle$
\end{example}

Our results can potentially apply to all models whose stationary distribution is known to a certain extent, as in Example~\ref{ex:auto}, provided that the assumptions of Corollary~\ref{cor:main} or the more general Corollary~\ref{cor:relaxed} are satisfied. Much effort has been devoted to calculating the exact form of the stationary distribution for families of stochastic reaction networks, a classical example being \cite{AndProdForm}. The results of \cite{AndProdForm} have been extended in \cite{anderson2016product}. In \cite{hong2021derivation}, Theorem~\ref{thm:def0} stated below has been extended and systematically utilized to calculate the stationary distributions for a variety of models, notably auto-catalytic models, and the computational package CASTANET is developed. Further characterizations of the stationary distributions of auto-catalytic networks are given in \cite{hoessly2019stationary}, and even further analysis on the form of the stationary distributions for families of stochastic reaction systems is in the recent paper \cite{hoessly2021stationary}.
The main scope of this section, however, is not to investigate the full applicability of our results in the setting of stochastic reaction networks, but to merely suggest examples of applications. Hence, for the sake of simplicity, we do not state here the results obtained in \cite{anderson2016product, hong2021derivation, hoessly2019stationary, hoessly2021stationary}, and limit our applications to the classical Theorem~\ref{thm:def0}, stated below, which is a slightly stronger version of the classical result contained in \cite{AndProdForm}.

\subsubsection{Reaction networks related to complex-balancing}

A special case of particular relevance in reaction network theory is given by \emph{complex-balanced mass-action systems}: these are mass-action systems $(\Sp,\C,\Re,\K)$ for which there exists a positive vector $c\in \R^d_{>0}$ satisfying
\begin{equation*}%\label{eq:cb}
    \sum_{\substack{y'\in\C:\\ y \to y' \in \Re}} \kappa_{y\to y'}\prod_{i=1}^d c_i^{y_i} = \sum_{\substack{y'\in\C:\\ y' \to y \in \Re}} \kappa_{y'\to y}\prod_{i=1}^d c_i^{y'_i}, 
\end{equation*}
for each fixed complex $y\in\C$. In this case, the vector $c$ is termed a \emph{complex-balanced equilibrium}. The study of complex-balanced networks dates back to \cite{Feinberg72,HornJack72}, where necessary and sufficient conditions for its existence are given, and has an important role in the development of reaction network theory. We give a minimal introduction in Section~\ref{sec:det}, where we explain why the vector $c$ is called an \emph{equilibrium}.

We will apply our theory to a class of stochastic reaction systems that are related to complex-balancing, but are not necessarily complex-balanced themselves. To this aim, the only result we need from the literature is the following one, which is proven in \cite{AT2019} and extends previous work done in \cite{AndProdForm, anderson2018non, CW2016}.

\begin{thm} \label{thm:def0} Let $(\Sp,\C,\Re,\K)$ be a mass-action system with a complex-balanced equilibrium $c \in \R^d_{> 0}$. Consider the stochastic process $\{X(t),t\in\R_{>0}\}$ given by the stochastic reaction system $(\Sp,\C,\Re,\{\lambda_{y\to y'}\})$ where the rate functions are of the form
\begin{equation}\label{eq:theta_def}
 \lambda_{y\to y'}(x)=\kappa_{y\to y'}\prod_{i=1}^d\prod_{j=0}^{y_i-1}\theta_i(x_i-j)
\end{equation}
for the same rate constants $\kappa_{y\to y'}$ as in $(\Sp,\C,\Re,\K)$ and some functions $\theta_i:\mathbb{Z}\to\R_{\geq0}$ satisfying
\begin{align}\label{eq:theta=0}
 \theta_i(n)&=0\quad\text{if and only if }n\leq 0\\
 \label{eq:theta_infty}
 \lim_{j\to\infty}\theta_i(j)&=\infty.
\end{align}
Then, $\{X(t), t\in\R_{>0}\}$ is non-explosive and admits the stationary distribution
\begin{equation}\label{eq:stat}
    \pi(x)=M\prod_{i = 1}^d \frac{c_i^{x_i}}{\prod_{j=1}^{x_i}\theta_i(j)}, \quad \text{for each $x\in \Z_{\ge0}^d$},
\end{equation}
where $M$ is a positive normalizing constant and the empty product is to be considered equal to 1.
\end{thm}

\begin{rem}\label{rem:ma}
 Note that stochastic mass-action kinetics satisfies \eqref{eq:theta_def}, \eqref{eq:theta=0}, and \eqref{eq:theta_infty} with the choice \begin{equation}\label{eq:mak}
 \theta_i(j)=\begin{cases}
 j&\text{if }j\geq1\\
 0&\text{otherwise}
 \end{cases}.
 \end{equation}
 In this case, \eqref{eq:stat} is a product-form Poisson distribution.
\end{rem}

% \begin{rem}
%  Theorem~\ref{thm:def0} has been extended in \cite{anderson2016product}. In \cite{hong2021derivation}, Theorem~\ref{thm:def0} has been extended and systematically utilized to calculate the stationary distributions in a variety of models, notably auto-catalytic models, and the computational package CASTANET is developed. Further characterisations of the stationary distributions of auto-catalytic networks are given in \cite{hoessly2019stationary}, and even further analysis on stationary distribution forms is in the recent paper \cite{hoessly2021stationary}. Our results can potentially apply to all models whose stationary distribution is known to a certain extent, provided that the assumptions of Corollary~\ref{cor:main} or the more general Corollary~\ref{cor:relaxed} are satisfied. However, the main scope of this section is not to investigate the full applicability of our results in the setting of stochastic reaction networks, but to merely suggest examples of applications. Hence, for the sake of simplicity, we do not state here the results obtained in \cite{anderson2016product, hong2021derivation, hoessly2019stationary, hoessly2021stationary}, and limit our applications to the classical Theorem~\ref{thm:def0}.
% \end{rem}

We can state the following results.

\begin{thm}\label{thm:cb_general_inout}
Let $(\Sp,\C,\Re,\K)$ be a mass-action system with a complex-balanced equilibrium $c \in \R^d_{> 0}$, such that
\begin{align*}
%\label{eq:all in and out}
    \{X_i \to 0 : X_i \in \Sp\} \bigcup   \{0 \to X_i: X_i \in \Sp\} \subseteq \Re.
    \end{align*}
Consider the stochastic process $\{X(t),t\in\R_{>0}\}$ given by the stochastic reaction system $(\Sp,\C,\Re,\{\lambda_{y\to y'}\})$ where the rate functions are of the form
\begin{equation*}
 \lambda_{y\to y'}(x)=\kappa_{y\to y'}\prod_{i=1}^d\prod_{j=0}^{y_i-1}\theta_i(x_i-j)
\end{equation*}
for the same rate constants $\kappa_{y\to y'}$ as in $(\Sp,\C,\Re,\K)$ and some functions $\theta_i:\mathbb{Z}\to\R_{\geq0}$ satisfying
\begin{equation*}
\begin{aligned}
 &\theta_i(n)=0&&\text{if }n\leq 0\\
 &\theta_i(n)\geq n^{\beta}&&\text{if }n> 0
 \end{aligned}
\end{equation*}
for some $\beta\in\R_{>0}$. Then, $\{X(t), t\in\R_{>0}\}$ is exponentially ergodic in $L^2$.
\end{thm}
\begin{proof}
 First of all, note that the presence of the reactions $0 \to X_i$ and $X_i\to 0$ for each $X_i\in\Sp$ guarantees the irreducibility of the state space $\Z^d_{\geq0}$. Theorem~\ref{thm:def0} implies the process $\{X(t), t\in\R_{>0}\}$ is non-explosive and has a unique stationary distribution $\pi$ satisfying
 \begin{equation*}
     \frac{\pi(x)}{\pi(x-e_i)}=\frac{c_i}{\theta_i(x_i)}\leq \frac{c_i}{x_i^\beta},
 \end{equation*}
 hence for any $0<\alpha<\beta$ and for any large enough $x_i$ we have
  \begin{equation*}
     \frac{\pi(x)}{\pi(x-e_i)}\leq \frac{1}{x_i^\alpha}.
 \end{equation*}
 The first condition of Corollary~\ref{cor:main} holds, and the second condition is satisfied by the presence of the reactions $0 \to X_i$ and $X_i\to 0$ for each $X_i\in\Sp$:
 \begin{equation*}
 \inf_{x\in\Z^d_{\geq0}}q(x,x+e_i)\geq \kappa_{0\to X_i}\quad\text{and}\quad\inf_{\substack{x\in\Z^d_{\geq0} \\ x_i\geq1}}q(x,x-e_i)\geq \inf_{x_i\geq1}\kappa_{X_i\to 0}\theta_i(x_i)=\kappa_{X_i\to 0}.
 \end{equation*}
 The proof is then concluded by the application of Corollary~\ref{cor:main}
\end{proof}

A natural corollary follows from the application of Theorem~\ref{thm:cb_general_inout} to the particular case of mass-action kinetics.

\begin{cor}\label{cor:cb_inout}
Let $(\Sp,\C,\Re,\K)$ be a mass-action system with a complex-balanced equilibrium $c \in \R^d_{> 0}$, such that
\begin{align*}%\label{eq:all in and out}
    \{X_i \to \emptyset : X_i \in \Sp\} \bigcup   \{\emptyset \to X_i: X_i \in \Sp\} \subseteq \Re.
    \end{align*}
Then, the associated process is exponentially ergodic in $L^2$.
Moreover, for any fixed $\varepsilon>0$ there exists $C_\varepsilon\in\R_{>0}$ such that $\tau_x^\varepsilon \le C_\varepsilon \|x\|\ln(\|x\|)$, for any $x\in\Z^d_{\geq0}\setminus\{0\}$.
\end{cor}
\begin{proof}
 The proof follows from direct application of Theorem~\ref{thm:cb_general_inout} to the specific choice of rate functions \eqref{eq:mak}
 as discussed in Remark~\ref{rem:ma}. The bound on the mixing time follows from the conclusions of Theorem~\ref{thm:main for poincare} and from the fact that $\pi$ is a product-form Poisson distribution.
\end{proof}

Before applying the more general Corollary~\ref{cor:relaxed} to stochastic reaction networks, we show examples of applications of Theorems~\ref{thm:cb_general_inout} and Corollary~\ref{cor:cb_inout}

\begin{example}\label{ex:open-cxb}
Consider the stochastic mass-action system $(\Sp,\C,\Re, \K)$ associated with the following reaction network, where the rate constants are written next to the associated reaction:
\begin{equation}\label{eq:ex_network}
    \begin{tikzpicture}[baseline={(current bounding box.center)}]
    \node[state] (X1) at (-2,1) {$X_1$};
    \node[state] (0)  at (0,0) {$0$};
    \node[state] (X2)  at (2,1) {$X_2$};
    \node[state] (A)  at (-2,-1) {$2X_1+X_2$};
    \node[state] (B)  at (2,-1) {$3X_1+2X_2$};
    \path[->]
     (0) edge[bend left] node {1} (X1)
     (X1) edge[bend left] node {1} (0)
     (0) edge[bend left] node {1} (X2)
     (X2) edge[bend left] node {1} (0)
     (A) edge node {1} (B)
     (0) edge node {1} (A)
     (B) edge node {1} (0);
   \end{tikzpicture}
\end{equation}
It is easy to check that $c = (1,1)$ is a complex-balanced equilibrium. Furthermore 
\begin{align*}
\{X_1\to \emptyset, X_2\to \emptyset\} \bigcup    \{\emptyset \to X_1, \emptyset \to X_2\} \subset \Re.
\end{align*}
Hence, by Corollary~\ref{cor:cb_inout} the model is exponentially ergodic in $L^2$ and for any $\varepsilon\in\R_{>0}$ we have $\tau_x^\varepsilon = O(\|x\|\ln(\|x\|))$. \hfill $\triangle$
\end{example}
\begin{example}
Consider a stochastic mass-action system $(\Sp,\C,\Re, \K)$ associated with the following reaction network:
\begin{equation*}
    \begin{tikzpicture}[baseline={(current bounding box.center)}]
    \node[state] (X1) at (2,0) {$X_2$};
    \node[state] (0)  at (0,0) {$0$};
    \node[state] (X2)  at (0,1.5) {$X_1$};
    \node[state] (2X1)  at (4,0) {$2X_2$};
    \path[->]
     (0) edge[bend left] node {} (X1)
     (X1) edge[bend left] node {} (0)
     (0) edge[bend left] node {} (X2)
     (X2) edge[bend left] node {} (0)
     (2X1) edge node {} (X1)
     (X2) edge node {} (X1);
   \end{tikzpicture}
\end{equation*}
The model describes coexisting individuals of two types ($X_1$ and $X_2$). The individuals of type $X_2$ fight against each other (modelled by $2X_2\to X_2$) and individuals of type $X_1$ naturally become individuals of type $X_2$. Such behaviour is similar, for example, to those of territorial fishes, where individuals of type $X_1$ can be interpreted as young offspring eventually becoming adult. We include inflows and outflows of both types of individuals, which makes sense if for example we study a portion of water that fishes can freely enter or exit.

Since the network is not strongly connected (or, in the language of reaction network theory, \emph{weakly reversible}), it is known that it cannot be complex-balanced for any choice of rate constants \cite{HornJack72, Horn72, Feinberg72}. However, the associated stochastic process has the same distribution as the one associated with the network 
\begin{equation}\label{eq:cb_version}
    \begin{tikzpicture}[baseline={(current bounding box.center)}]
    \node[state] (X1) at (2,0) {$X_2$};
    \node[state] (0)  at (0,0) {$0$};
    \node[state] (X2)  at (0,1.5) {$X_1$};
    \path[->]
     (0) edge[bend left] node {} (X1)
     (X1) edge[bend left] node {} (0)
     (0) edge[bend left] node {} (X2)
     (X2) edge[bend left] node {} (0)
     (X2) edge node {} (X1);
   \end{tikzpicture}
\end{equation}
with the following choice of kinetics:
\begin{equation*}
\arraycolsep=1.4pt
\begin{array}{rclcrcl}
    \lambda_{0\to X_1}(x)&=&\kappa_{0\to X_1}&\phantom{space}&
    \lambda_{0\to X_2}(x)&=&\kappa_{0\to X_2}\\
    \lambda_{X_1\to 0}(x)&=&\kappa_{X_1\to 0}\theta_1(x_1)&&
    \lambda_{X_2\to 0}(x)&=&\kappa_{X_2\to 0}\theta_2(x_2)\\
    \lambda_{X_1\to X_2}(x)&=&\kappa_{X_1\to X_2}\theta_1(x_1)&&
\end{array}
\end{equation*}
where
\begin{equation*}
    \theta_1(n)=\begin{cases}
    n &\text{if }n\geq0\\
    0 &\text{otherwise}
    \end{cases}
    \quad\text{and}\quad
    \theta_2(n)=\begin{cases}
    n+\frac{\kappa_{2X_2\to X_2}}{\kappa_{X_2\to 0}}n(n-1)&\text{if }n\geq0\\
    0 &\text{otherwise}
    \end{cases}.
\end{equation*}
By classical theory developed in \cite{HornJack72, Horn72, Feinberg72} we know that \eqref{eq:cb_version}, as a mass-action system, is complex-balanced for any choice of rate constants (it is weakly reversible and its deficiency is 0). Hence, we can apply Theorem~\ref{thm:cb_general_inout} and conclude that our original model is exponentially ergodic.\hfill$\triangle$
\end{example}
We now relax Theorem~\ref{thm:cb_general_inout} and Corollary~\ref{cor:cb_inout} to allow for the birth and death of certain species to be catalyzed by other species. The proofs are a direct application of Corollary~\ref{cor:relaxed}, exactly as the proofs of Theorem~\ref{thm:cb_general_inout} and Corollary~\ref{cor:cb_inout} were a direct application of Corollary~\ref{cor:main}, so we will omit them.
\begin{thm}\label{thm:cb_general_inout_relaxed}
        Let $(\Sp,\C,\Re,\K)$ be a mass-action system with a complex-balanced equilibrium $c \in \R^d_{> 0}$, such that:
    \begin{enumerate}
        \item there exists a non-empty subset $\Sp_0 \subset \Sp$ so that for each $i\in \Sp_0$ we have 
        \begin{align*}
         \{X_i \to \emptyset, \emptyset\to X_i\} \subseteq \Re;
        \end{align*}

        \item there exist disjoint subsets of $\Sp_1,\dots,\Sp_m\subset \Sp$ so that $\{\Sp_0,\Sp_1, \dots, \Sp_m\}$ is a partition of $\Sp$, such that for each $i \in J_\ell$ with $\ell \ge 1$, we have
    \begin{enumerate}
         \item[(i)] $N^+_{ji}X_j\to N^+_{ji}X_j + X_i $ for some $j\in \Sp_0 \cup\cdots \cup  \Sp_{\ell-1}$ with some $N^+_{ji}\in \mathbb Z_{\ge 0}$, and 
            \item[(ii)] $N^-_{ji}X_j+X_i\to N^-_{ji}X_j$ for some $j\in \Sp_0 \cup\cdots \cup \Sp_{\ell-1}$ with some $N^-_{ji}\in \mathbb Z_{\ge 0}$.
    \end{enumerate}
    \end{enumerate}
Consider the stochastic process $\{X(t),t\in\R_{>0}\}$ given by the stochastic reaction system $(\Sp,\C,\Re,\{\lambda_{y\to y'}\})$ where the rate functions are of the form
\begin{equation*}
 \lambda_{y\to y'}(x)=\kappa_{y\to y'}\prod_{i=1}^d\prod_{j=0}^{y_i-1}\theta_i(x_i-j)
\end{equation*}
for the same rate constants $\kappa_{y\to y'}$ as in $(\Sp,\C,\Re,\K)$ and some functions $\theta_i:\mathbb{Z}\to\R_{\geq0}$ satisfying
\begin{equation*}
\begin{aligned}
 &\theta_i(n)=0&&\text{if }n\leq 0\\
 &\theta_i(n)\geq n^{\beta}&&\text{if }n> 0
 \end{aligned}
\end{equation*}
for some $\beta\in\R_{>0}$. Then, $\{X(t), t\in\R_{>0}\}$ is exponentially ergodic in $L^2$.
\end{thm}
\begin{cor}\label{cor:cb_inout_relaxed}
        Let $(\Sp,\C,\Re,\K)$ be a mass-action system with a complex-balanced equilibrium $c \in \R^d_{> 0}$, such that:
    \begin{enumerate}
        \item there exists a non-empty subset $\Sp_0 \subset \Sp$ so that for each $i\in \Sp_0$ we have 
        \begin{align*}
         \{X_i \to \emptyset, \emptyset\to X_i\} \subseteq \Re;
        \end{align*}

        \item there exist disjoint subsets of $\Sp_1,\dots,\Sp_m\subset \Sp$ so that $\{\Sp_0,\Sp_1, \dots, \Sp_m\}$ is a partition of $\Sp$, such that for each $i \in J_\ell$ with $\ell \ge 1$, we have
    \begin{enumerate}
         \item[(i)] $N^+_{ji}X_j\to N^+_{ji}X_j + X_i $ for some $j\in \Sp_0 \cup\cdots \cup  \Sp_{\ell-1}$ with some $N^+_{ji}\in \mathbb Z_{\ge 0}$, and 
            \item[(ii)] $N^-_{ji}X_j+X_i\to N^-_{ji}X_j$ for some $j\in \Sp_0 \cup\cdots \cup \Sp_{\ell-1}$ with some $N^-_{ji}\in \mathbb Z_{\ge 0}$.
    \end{enumerate}
    \end{enumerate}
Then, the associated process is exponentially ergodic in $L^2$.
Moreover, for any fixed $\varepsilon>0$ there exists $C_\varepsilon\in\R_{>0}$ such that $\tau_x^\varepsilon \le C_\varepsilon \|x\|\ln(\|x\|)$, for any $x\in\Z^d_{\geq0}\setminus\{0\}$.
\end{cor}

\begin{example}\label{ex: key example of crn}
    The Markov chain given in Example \ref{ex:key example}  models the copy numbers of chemical species $X_1$ and $X_2$ involved in the reaction system 
   \begin{equation*}
   \begin{tikzpicture}[baseline={(current bounding box.center)}]
   \node[state] (X1+X2) at (-2.7,0) {$X_1+X_2$};
   \node[state] (X2)  at (0,0) {$X_2$};
   \node[state] (0)  at (2,0) {$0$};
    \path[->]
     ([yshift=-0.3em]X2.west) edge[bend left] node[below] {$\kappa_1$} ([yshift=-0.3em]X1+X2.east)
     ([yshift=0.3em]X1+X2.east) edge[bend left] node {$\kappa_2$} ([yshift=0.3em]X2.west)
     ([yshift=-0.3em]0.west) edge[bend left] node[below] {$\kappa_3$} ([yshift=-0.3em]X2.east)
     ([yshift=0.3em]X2.east) edge[bend left] node {$\kappa_4$} ([yshift=0.3em]0.west);
   \end{tikzpicture}
\end{equation*}
with mass action kinetics.
The model is complex-balanced for any choice of rate constants \cite{Horn72, Feinberg72, HornJack72}. Hence, the Markov chain is exponentially ergodic since the conditions in Corollary~\ref{cor:cb_inout_relaxed} hold. Note that $X_2$ catalyzes the birth and death of $X_1$.
\hfill $\triangle$
\end{example}

% \section*{Statements and Declarations}

% The authors declare there is no conflict of interest.

% \bibliographystyle{plain}
% \bibliography{res}

\begin{thebibliography}{10}

\bibitem{agazzi2018large}
Andrea Agazzi, Amir Dembo, and Jean-Pierre Eckmann.
\newblock Large deviations theory for markov jump models of chemical reaction
  networks.
\newblock {\em The Annals of Applied Probability}, 28(3):1821--1855, 2018.

\bibitem{agazzi2018geometry}
Andrea Agazzi, Amir Dembo, and Jean-Pierre Eckmann.
\newblock On the geometry of chemical reaction networks: Lyapunov function and
  large deviations.
\newblock {\em Journal of Statistical Physics}, 172(2):321--352, 2018.

\bibitem{AndGAC_ONE2011}
David~F. Anderson.
\newblock A proof of the global attractor conjecture in the single linkage
  class case.
\newblock {\em SIAM J. Appl. Math}, 71(4):1487 -- 1508, 2011.

\bibitem{anderson2020stochastically}
David~F. Anderson, Daniele Cappelletti, and Jinsu Kim.
\newblock Stochastically modeled weakly reversible reaction networks with a
  single linkage class.
\newblock {\em Journal of Applied Probability}, 57(3):792--810, 2020.

\bibitem{anderson2020tier}
David~F. Anderson, Daniele Cappelletti, Jinsu Kim, and Tung~D. Nguyen.
\newblock Tier structure of strongly endotactic reaction networks.
\newblock {\em Stochastic Processes and their Applications},
  130(12):7218--7259, 2020.

\bibitem{anderson2018non}
David~F. Anderson, Daniele Cappelletti, Masanori Koyama, and Thomas~G. Kurtz.
\newblock Non-explosivity of stochastically modeled reaction networks that are
  complex balanced.
\newblock {\em Bulletin of mathematical biology}, 80(10):2561--2579, 2018.

\bibitem{Nonexplo}
David~F. Anderson, Daniele Cappelletti, Masanori Koyama, and Thomas~G. Kurtz.
\newblock Non-explosivity of stochastically modeled reaction networks that are
  complex balanced.
\newblock {\em Bull. Math. Bio.}, 80(10):2561--2579, 2018.

\bibitem{anderson2016product}
David~F. Anderson and Simon~L. Cotter.
\newblock Product-form stationary distributions for deficiency zero networks
  with non-mass action kinetics.
\newblock {\em Bulletin of mathematical biology}, 78(12):2390--2407, 2016.

\bibitem{AndProdForm}
David~F. Anderson, Gheorghe Craciun, and Thomas~G. Kurtz.
\newblock {Product-form stationary distributions for deficiency zero chemical
  reaction networks}.
\newblock {\em Bull. Math. Biol.}, 72(8):1947--1970, 2010.

\bibitem{AEJ2014}
David~F. Anderson, Germ{\'{a}}n~A Enciso, and Matthew~D Johnston.
\newblock {Stochastic analysis of biochemical reaction networks with absolute
  concentration robustness}.
\newblock {\em Royal Society Interface}, 11:20130943, 2014.

\bibitem{anderson2018some}
David~F. Anderson and Jinsu Kim.
\newblock Some network conditions for positive recurrence of stochastically
  modeled reaction networks.
\newblock {\em SIAM Journal on Applied Mathematics}, 78(5):2692--2713, 2018.

\bibitem{mixing_And_Kim}
David~F. Anderson and Jinsu Kim.
\newblock Mixing times for two classes of stochastically modeled reaction
  networks.
\newblock {\em Mathematical Biosciences and Engineering}, 20(3):4690--4713,
  2023.

\bibitem{AK2015}
David~F. Anderson and Thomas~G. Kurtz.
\newblock {\em {Stochastic analysis of biochemical systems}}, volume 1.2 of
  {\em Stochastics in Biological Systems}.
\newblock Springer International Publishing, Switzerland, 1 edition, 2015.

\bibitem{AT2019}
David~F. Anderson and Tung~D. Nguyen.
\newblock Results on stochastic reaction networks with non-mass action
  kinetics.
\newblock {\em Mathematical Biosciences and Engineering}, 16(4):2118--2140,
  2019.

\bibitem{superlya2012}
Avanti Athreya, Tiffany Kolba, and Jonathan~C. Mattingly.
\newblock {Propagating Lyapunov functions to prove noise-induced
  stabilization}.
\newblock {\em Electron J. Probab.}, 17(96):1--38, 2012.

\bibitem{bakry2008rate}
Dominique Bakry, Patrick Cattiaux, and Arnaud Guillin.
\newblock Rate of convergence for ergodic continuous markov processes: Lyapunov
  versus poincar{\'e}.
\newblock {\em Journal of Functional Analysis}, 254(3):727--759, 2008.

\bibitem{bakry2014analysis}
Dominique Bakry, Ivan Gentil, Michel Ledoux, et~al.
\newblock {\em Analysis and geometry of Markov diffusion operators}, volume
  103.
\newblock Springer, 2014.

\bibitem{berestycki2009eight}
Nathan{\"e}l Berestycki.
\newblock Eight lectures on mixing times.
\newblock {\em Typed Lecture Notes, Department of Statistics, Cambridge
  University}, 2009.

\bibitem{bibbona2020stationary}
Enrico Bibbona, Jinsu Kim, and Carsten Wiuf.
\newblock Stationary distributions of systems with discreteness-induced
  transitions.
\newblock {\em Journal of The Royal Society Interface}, 17(168):20200243, 2020.

\bibitem{cappelletti2018graphically}
Daniele Cappelletti and Badal Joshi.
\newblock Graphically balanced equilibria and stationary measures of reaction
  networks.
\newblock {\em SIAM Journal on Applied Dynamical Systems}, 17(3):2146--2175,
  2018.

\bibitem{CW2016}
Daniele Cappelletti and Carsten Wiuf.
\newblock Product-form {P}oisson-like distributions and complex balanced
  reaction systems.
\newblock {\em SIAM Journal on Applied Mathematics}, 76(1):411--432, 2016.

\bibitem{MR1106707}
Mu~Fa Chen.
\newblock Exponential {$L^2$}-convergence and {$L^2$}-spectral gap for {M}arkov
  processes.
\newblock {\em Acta Math. Sinica (N.S.)}, 7(1):19--37, 1991.
\newblock A Chinese summary appears in Acta Math. Sinica {{\bf{3}}5} (1992),
  no. 2, 286.

\bibitem{CraciunShiu09}
Gheorghe Craciun, Alicia Dickenstein, Anne Shiu, and Bernd Sturmfels.
\newblock {Toric Dynamical Systems}.
\newblock {\em Journal of Symbolic Computation}, 44:1551--1565, 2009.

\bibitem{dickenstein2011far}
Alicia Dickenstein and Mercedes P{\'e}rez~Mill{\'a}n.
\newblock How far is complex balancing from detailed balancing?
\newblock {\em Bulletin of mathematical biology}, 73(4):811--828, 2011.

\bibitem{down1995exponential}
Douglas Down, Sean~P. Meyn, and Richard~L. Tweedie.
\newblock Exponential and uniform ergodicity of markov processes.
\newblock {\em The Annals of Probability}, 23(4):1671--1691, 1995.

\bibitem{ethier2009markov}
Stewart~N. Ethier and Thomas~G. Kurtz.
\newblock {\em Markov processes: characterization and convergence}, volume 282.
\newblock John Wiley \& Sons, 2009.

\bibitem{Feinberg72}
M~Feinberg.
\newblock {Complex balancing in general kinetic systems}.
\newblock {\em Arch. Rational Mech. Anal.}, 49:187--194, 1972.

\bibitem{feliu201868}
Elisenda Feliu, Daniele Cappelletti, and Carsten Wiuf.
\newblock Node balanced steady states: Unifying and generalizing complex and
  detailed balanced steady states.
\newblock {\em Mathematical Biosciences}, 301:68–82, 2018.

\bibitem{hoessly2021stationary}
Linard Hoessly.
\newblock Stationary distributions via decomposition of stochastic reaction
  networks.
\newblock {\em Journal of Mathematical Biology}, 82(7):67, 2021.

\bibitem{hoessly2019stationary}
Linard Hoessly and Christian Mazza.
\newblock Stationary distributions and condensation in autocatalytic reaction
  networks.
\newblock {\em SIAM Journal on Applied Mathematics}, 79(4):1173--1196, 2019.

\bibitem{hong2021derivation}
Hyukpyo Hong, Jinsu Kim, M~Ali Al-Radhawi, Eduardo~D Sontag, and Jae~Kyoung
  Kim.
\newblock Derivation of stationary distributions of biochemical reaction
  networks via structure transformation.
\newblock {\em Communications biology}, 4(1):620, 2021.

\bibitem{Horn72}
Friedrich Josef~Maria Horn.
\newblock {Necessary and sufficient conditions for complex balancing in
  chemical kinetics}.
\newblock {\em Arch. Rat. Mech. Anal.}, 49(3):172--186, 1972.

\bibitem{HornJack72}
Friedrich Josef~Maria Horn and Roy Jackson.
\newblock {General Mass Action Kinetics}.
\newblock {\em Arch. Rat. Mech. Anal.}, 47:81--116, 1972.

\bibitem{joshi2015detailed}
Badal Joshi.
\newblock A detailed balanced reaction network is sufficient but not necessary
  for its markov chain to be detailed balanced.
\newblock {\em Discrete and Continuous Dynamical Systems-B}, 20(4):1077--1105,
  2015.

\bibitem{Kelly1979}
Frank~P. Kelly.
\newblock {\em {Reversibility and stochastic networks}}.
\newblock John Wiley \& Sons Ltd., Chichester, 1979.

\bibitem{YuvalLevinMixing}
David~A. Leven and Yuval Peres.
\newblock {\em {Markov Chains and Mixing Times}}.
\newblock American Mathematical Society, 2009.

\bibitem{Liggett89}
Thomas~M. Liggett.
\newblock Exponential {$L_2$} convergence of attractive reversible nearest
  particle systems.
\newblock {\em The Annals of Probability}, pages 403--432, 1989.

\bibitem{ma2012introduction}
Zhi-Ming Ma and Michael R{\"o}ckner.
\newblock {\em Introduction to the theory of (non-symmetric) Dirichlet forms}.
\newblock Springer Science \& Business Media, 2012.

\bibitem{muller2020detailed}
Stefan M{\"u}ller and Badal Joshi.
\newblock Detailed balance {=} complex balance {+} cycle balance: A
  graph-theoretic proof for reaction networks and markov chains.
\newblock {\em Bulletin of mathematical biology}, 82(9):116, 2020.

\bibitem{NorrisMC97}
James Norris.
\newblock {\em {Markov Chains}}.
\newblock Cambridge University Press, 1997.

\bibitem{rockner2001weak}
Michael R{\"o}ckner and Feng-Yu Wang.
\newblock Weak poincar{\'e} inequalities and l2-convergence rates of markov
  semigroups.
\newblock {\em Journal of Functional Analysis}, 185(2):564--603, 2001.

\bibitem{saloff1997lectures}
Laurent Saloff-Coste.
\newblock Lectures on finite {M}arkov chains.
\newblock In P.~Bernard, editor, {\em Lectures on Probability Theory and
  Statistics}, pages 301--413. Springer, 1997.

\bibitem{Sh-F}
Guy Shinar and Martin Feinberg.
\newblock {Structural Sources of Robustness in Biochemical Reaction Networks}.
\newblock {\em Science}, 327(5971):1389--1391, 2010.

\bibitem{TK2001}
Yuichi Togashi and Kunihiko Kaneko.
\newblock Transitions induced by the discreteness of molecules in a small
  autocatalytic system.
\newblock {\em Phys. Rev. Lett.}, 86:2459--2462, Mar 2001.

\bibitem{Han16}
Ramon van Handel.
\newblock Probability in high dimension, December 21, 2016.
\newblock Lecture notes retrieved from
  \url{https://web.math.princeton.edu/~rvan/APC550.pdf}.

\bibitem{Whittle86}
Peter Whittle.
\newblock {\em {Systems in stochastic equilibrium}}.
\newblock John Wiley \& Sons, Inc., New York, NY, USA, 1986.

\bibitem{XuHansenWiuf2022}
Chuang Xu, Mads~Christian Hansen, and Carsten Wiuf.
\newblock Full classification of dynamics for one-dimensional continuous-time
  markov chains with polynomial transition rates.
\newblock {\em Advances in Applied Probability}, pages 1--35, 2022.

\end{thebibliography}

\appendix

\section{Technical results related to relevant operators}
 
In this brief section, we suppose that $X$ is an irreducible and positive-recurrent continuous-time Markov chain with countable state space $\S\subset \Z^d$ and stationary distribution $\pi$ that satisfies Assumption \ref{intcnd}. 
We let $P_t$ denote the semigroup operator on $L^2(\pi)$ associated with $X$  defined by $P_tf(x)=\mathbb{E}_x[f(X_t)]$,
and denote the infinitesimal generator by $\mathcal{A}$ and its domain in $L^2(\pi)$ by $Dom(\mathcal{A})$.

The first two results are standard (for example, see \cite{ethier2009markov}) but are required for our proofs and so are included for completeness.  In particular, they allow us to utilize the results from Chapter 1 of \cite{ethier2009markov} (especially Proposition 1.5, which we use throughout).  The third result is required for the proof of Theorem \ref{thm:adlkjfkd}.  It is straightforward, but we were not able to find it in the literature and so include it here.

\begin{lem}\label{lem:contraction}
  $P_t$ is a contraction on $L^2(\pi)$.  That is, for $f\in L^2(\pi)$, we have $\|P_t f\|_{L^2(\pi)}^2 \le \|f\|_{L^2(\pi)}^2$.
\end{lem}

\begin{lem}\label{lem:stongcts}
$P_t$  is strongly continuous on  $L^2(\pi)$.  That is, for $f \in L^2(\pi)$, we have $\lim_{t \to 0} P_t f = f$, in $L^2(\pi)$.
\end{lem}

\begin{lem}\label{lem:uniformconv}
 For $f \in Dom(\mathcal{A})$ with $\|f\|_\infty<\infty$, the sum
\[
\sum_{x\in\mathbb{S}}\pi(x)  P_tf(x)\,\mathcal{A}P_tf(x)
\]
converges uniformly over $t\ge0$.
\end{lem}

\begin{proof}
Let $K_n\subset K_{n+1}$ be an increasing sequence of compact sets with $\cup_{n} K_n = \mathbb{S}$.  Denote $f(t) =  \sum_{x\in\mathbb{S}}\pi(x)  P_tf(x)\,\mathcal{A}P_tf(x)$ and $f_n(t) = \sum_{x\in K_n}\pi(x)  P_tf(x)\,\mathcal{A}P_tf(x)$.  Also denote $a_n = \sqrt{\sum_{x\in K_n^c} \pi(x)}$, and note that $a_n \to 0$, as $n\to \infty$.  Then, because $|P_t f(x)| \le \|f\|_\infty$ and $\mathcal{A} P_t f = P_t \mathcal{A} f$ (see \cite[Proposition 1.5]{ethier2009markov}), we have
\begin{align*}
    |&f(t) - f_n(t)| \le  \sum_{x\in K_n^c}\pi(x) | P_tf(x)| \,|\mathcal{A}P_tf(x)| \\
    &\le \sqrt{ \sum_{x\in K_n^c}\pi(x) |P_t f(x)|^2}\sqrt{ \sum_{x\in K_n^c}\pi(x) |P_t \mathcal A  f(x)|^2}\tag{Cauchy-Schwarz}\\
    &\le \|f\|_\infty \sqrt{ \sum_{x\in K_n^c}\pi(x) }\sqrt{ \sum_{x\in \mathbb{S}}\pi(x) |P_t \mathcal A  f(x)|^2}\tag{$|P_t f(x)|\le \|f\|_\infty$ and converted $K_n^c$ to $\mathbb{S}$}\\
    &\le \left( \|f\|_\infty \|\mathcal{A} f\|_{L^2(\pi)}\right) \cdot a_n,
\end{align*}
where in the final inequality we used that $P_t$ is a contraction on $L^2(\pi)$.  Hence, the convergence is uniform in $t\ge 0$. 
\end{proof}

\section{Proof of Lemma \ref{lemma:compact_f}}
\label{appendix:proof_1234}

\begin{proof}[Proof of Lemma \ref{lemma:compact_f}]
Let $f\in \mathcal{B}_c$ and let $K\subset \mathbb{S}$ denote the support of $f$.  We write $|K|$ for the number of elements in $K$ (which is finite). Let 
\[
    C_1 = |K|\cdot \|f\|_\infty \max_{z \in K} \sup_{x\in \mathbb{S}\setminus \{z\}} q(x,z) < \infty,
\]
which is finite because of condition (iii) of Assumption \ref{intcnd}.  Next, let
\[
    C_2 =|K| \cdot  \|f\|_\infty \left(\max_{x\in K} q_x \right)<\infty,
\]
which is finite by condition (ii) of Assumption \ref{intcnd}.  Let $\hat{\mathcal{A}}$ be the linear operator defined by \eqref{Def:Generator}. If $f\in Dom(\mathcal{A})$ then $\mathcal{A}=\hat{\mathcal{A}}$, and that is what we aim to prove.
For $x \in \mathbb{S}$
\begin{align*}
    |\hat{\mathcal{A}} f(x)| &\le \sum_{z\in \mathbb{S}\setminus \{x\}} |f(z) - f(x)| q(x,z)\le \sum_{z\in K} |f(z)| q(x,z) + |f(x)| \sum_{z\in \mathbb{S}\setminus \{x\}} q(x,z)\le C_1 + C_2.
\end{align*}

 We now show that $\mathcal{B}_c\subset Dom(\mathcal{A})$. Let $f\in \mathcal{B}_c$.
By \cite[Theorem 4]{Nonexplo}, we have that 
\begin{align*}
    f(X(t)) - f(X(0)) - \int_0^t \hat{\mathcal{A}} f(X(s)) ds
\end{align*}
is a $\{\mathcal F^X_t\}$-martingale, where $\{\mathcal{F}_t^X\}$ is the filtration generated by $(X(t))_{t \ge 0}$. 
Therefore, for $x \in \mathbb{S}$,
\begin{align}\label{eq:expression1}
    \frac{\mathbb{E}_x[f(X(t))] - f(x)}{t} - \hat{\mathcal{A}}f(x) &= \mathbb{E}_x \left[ \frac{1}{t} \int_0^t( \hat{\mathcal{A}}f (X(s)) - \hat{\mathcal{A}}f(x) ) ds\right],
\end{align}
which converges to zero, as $t \to 0$, since $\hat{\mathcal{A}} f(X(s)) \to \hat{\mathcal{A}} f(x)$ almost surely by condition (ii) of Assumption~\ref{intcnd} and because the integrand is bounded by $2\|\hat{\mathcal{A}} f\|_\infty\leq 2(C_1+C_2)$.  Moreover, 
\begin{align*}
    \sum_{x\in \mathbb{S}} \pi(x) \left(  \hat{\mathcal{A}} f(x) - \frac{\mathbb{E}_x[f(X(t))] - f(x)}{t} \right)^2 \to 0, \quad \text{ as } t \to 0,
\end{align*}
since we have pointwise convergence (in $x\in \mathbb{S}$) from above and, again,  a   uniform bound on the  right-hand side of \eqref{eq:expression1} (from the uniform bound on $\hat{\mathcal{A}}f$).
\end{proof}

\section{Proof of Lemma \ref{lem:eff}} \label{appendix:proof_9808097}

\begin{proof}[Proof of Lemma \ref{lem:eff}]
Combining Assumption~\ref{intcnd} and the boundedness of $\|f\|_\infty$ demonstrates that the sum in the definition of $\EE^*(f)$ converges absolutely,
\begin{align*}
    \sum_{x,z\in \mathbb{S}}\big(f(x)-f(z)\big)^2\,\pi(x)\,q(x,z)\le 4 \|f\|_\infty^2 \sum_{x\in \mathbb{S}}  q_x \pi(x) <\infty,
\end{align*}
as does the sum for $\EE(f,f)$,
\begin{align*}
    \sum_{x,z \in \mathbb{S}}|f(x)(f(z)-f(x))\pi(x) q(x,z)|&\le 2\|f\|_\infty^2 \sum_{x\in \mathbb{S}} q_x \pi(x)<\infty.
\end{align*}
Hence, we can rearrange the terms of the sums to find the following,
\begin{align}
    2\EE(f,f)-2\EE^*(f) &= -2\sum_{x,z \in \mathbb{S}}f(x)(f(z)-f(x))\pi(x) q(x,z)-\sum_{x,z \in \mathbb{S}}(f(x)-f(z))^2 \pi(x) q(x,z) \notag\\
    %&= \sum_{ x,z \in \mathbb{S}}(f(x)-f(z))\left (f(x)+f(z) \right )\pi(x) q(x,z) \notag\\
    &= \sum_{x,z\in \mathbb{S}}f(x)^2\pi(x) q(x,z) - \sum_{x,z\in \mathbb{S}}f(z)^2\pi(x) q(x,z) \notag\\
    &= \sum_{x\in \mathbb{S}}f(x)^2 \pi(x)\sum_{z\in \mathbb{S}}  q(x,z) - \sum_{x,z \in \mathbb{S}}f(z)^2\pi(x) q(x,z) \notag\\
    &= \sum_{x\in \mathbb{S}}f(x)^2 \sum_{z\in \mathbb{S}} \pi(z) q(z,x) - \sum_{x,z \in \mathbb{S}}f(z)^2\pi(x) q(x,z) \label{eq:key finding}\\
    &= \sum_{x,z \in \mathbb{S}}f(x)^2 \pi(z) q(z,x) - \sum_{x,z \in \mathbb{S}}f(z)^2\pi(x) q(x,z)=0, \notag
\end{align} 
where to get the equality \eqref{eq:key finding} we used \eqref{eq:equilibrium}.
Hence $\EE(f,f)=\EE^*(f)$.
\end{proof}

\section{Deterministic model and complex-balanced systems}\label{sec:det}
Given a reaction network $(\Sp,\C,\Re)$, a \emph{(deterministic) kinetics} is an assignment of a function $\lambda_{y\to y'}: \R^d_{\ge0} \to \R_{\ge0}$ to each reaction $y\to y'\in \Re$.  We then call $(\Sp,\C,\Re,\{\lambda_{y\to y'}\})$ a \emph{deterministic reaction system}, and the associated ordinary differential equation governing the dynamics of the model is 
\begin{align}\label{eq:DetODE}
x(t) = x(0) + \sum_{y\to y'\in \Re} (y'-y) \int_0^t \lambda_{y\to y'}(x(s)) ds,    
\end{align}
where we recall that $y,y'$ are vectors in $\Z^d_{\ge 0}$.  Note that $y'-y$ is the vector describing the net gain of each species due to one instance of the reaction $y\to y'$. 

A popular choice of kinetics is given by \emph{(deterministic) mass-action kinetics}, where the form of the kinetics is given by
\[
    \lambda_{y\to y'}^{D}(x)=\kappa_{y\to y'} \prod_{i=1}^d x_i^{y_i} := \kappa_{y\to y'} x^y,
\]
for some positive constant $\kappa_{y\to y'}$, called a \emph{reaction constant}, where we take $0^0=1$ and where we define the notation $x^y$ above.  We denote $\K = \{\kappa_{y\to y'}\}$ and in the case of a deterministic mass-action system we write the system as $(\Sp,\C,\Re,\K)$ as opposed to $(\Sp,\C,\Re,\{\lambda_{y\to y'}\})$.  In the reaction graphs, the rate constants are typically placed next to the reaction arrow as in  $y\xrightarrow{\kappa_{y\to y'}} y'$ when we need to incorporate them. See Figure \ref{figure:EnvZ_model1} for an example of this methodology.
 
A fixed-point of \eqref{eq:DetODE} is said to be a \emph{complex-balanced equilibrium} if, at equilibrium, each complex has the following property: the total flux into the complex (according to the reaction graph) is equal to the total flux out of the complex (according to the reaction graph).  

\begin{defn}\label{def:CB}
A deterministic mass-action system $(\Sp,\C,\Re,\K)$ is \emph{complex-balanced} if the system \eqref{eq:DetODE} 
admits a positive vector $c \in \R^d_{> 0}$ such that for each complex $y\in \C$,
\begin{equation}\label{eq:cb}
    \sum_{\substack{y'\in \C \\ y\to y' \in \Re}} \kappa_{y\to y'}c^{y} = \sum_{\substack{y'\in \C \\ y'\to y \in \Re}} \kappa_{y'\to y}c^{y'}. 
\end{equation}
\end{defn}
It is shown in \cite{HornJack72} that a vector $c$ satisfying \eqref{eq:cb} is necessarily a steady state of \eqref{eq:DetODE} and is called a \emph{complex-balanced equilibrium}. In \cite{Feinberg72, HornJack72, Horn72} it is further shown that if a complex-balanced equilibrium exists, then necessarily all positive steady states of the mass-action system are complex-balanced. Sufficient and necessary condition for complex-balancing are also given, and the stability of complex-balanced equilibria is studied.

\end{document}